\RequirePackage[l2tabu, orthodox]{nag}
\documentclass[a4paper]{amsart}
\usepackage{amscd}

\usepackage{amsmath, amsfonts,amssymb,amsthm,mathrsfs,xypic}
\usepackage{graphicx}
\DeclareFontFamily{U}{mathx}{\hyphenchar\font45}
\DeclareFontShape{U}{mathx}{m}{n}{
      <5> <6> <7> <8> <9> <10>
      <10.95> <12> <14.4> <17.28> <20.74> <24.88>
      mathx10
      }{}
\DeclareSymbolFont{mathx}{U}{mathx}{m}{n}
\DeclareFontSubstitution{U}{mathx}{m}{n}
\DeclareMathAccent{\widecheck}{0}{mathx}{"71}
\DeclareMathAccent{\wideparen}{0}{mathx}{"75}

\xyoption{curve}

\usepackage[all]{xy}
\usepackage{microtype}

\numberwithin{equation}{section}

\theoremstyle{plain}
\newtheorem{theorem}[equation]{Theorem}
\newtheorem{proposition}[equation]{Proposition}
\newtheorem{lemma}[equation]{Lemma}
\newtheorem{corollary}[equation]{Corollary}

\theoremstyle{definition}
\newtheorem{definition}[equation]{Definition}
\newtheorem{example}[equation]{Example}

\theoremstyle{remark}
\newtheorem{remark}[equation]{Remark}

\newcommand{\Coker}{\operatorname{Coker}}
\newcommand{\Ext}{\operatorname{Ext}}
\newcommand{\id}{\operatorname{id}}
\newcommand{\Hom}{\operatorname{Hom}}
\newcommand{\Ker}{\operatorname{Ker}}

\newcommand{\Ho}{\mathrm{Ho}}
\newcommand{\R}{\mathrm{R}}
\newcommand{\rmL}{\mathrm{L}}
\newcommand{\rmD}{\mathrm{D}}
\newcommand{\ul}{\underline}
\newcommand{\xrto}{\xrightarrow}

\def\A{\mathcal A}

\def\C{\mathcal C}
\def\D{\mathcal D}
\def\E{\mathcal E}
\def\F{\mathcal F}
\def\G{\mathcal G}

\def\M{\mathcal M}

\def\T{\mathcal T}
\def\H{\mathcal H}

\def\X{\mathcal X}

\begin{document}

\title[\tiny{A homotopy theory of additive categories}]{A homotopy theory of additive categories with suspensions}
\author [\tiny{Zhi-Wei Li}] {Zhi-Wei Li}

\date{\today}
\thanks{The author was supported by National Natural Science Foundation
of China (No.s 11671174 and 11571329).}

\email{zhiweili@jsnu.edu.cn}
\maketitle


\maketitle
\begin{center}
\tiny{School of Mathematics and Statistics, \ \ Jiangsu Normal University \\
Xuzhou 221116, Jiangsu, PR China.}
\end{center}
\begin{abstract}
We develop a homotopy theory for additive categories endowed with endofunctors, analogous to the concept of a model structure. We use it to construct the homotopy theory of a Hovey triple (which consists of two compatible complete cotorsion pairs) in an arbitrary exact category. We show that the homotopy category of an exact model structure (in the sense of Hovey) in a weakly idempotent complete exact category is equivalent to the subfactor category of cofibrant-fibrant objects as pre-triangulated categories.
\end{abstract}

\setcounter{tocdepth}{1}

\section{Introduction}

The starting point to ``do homotopy theory" is the localization of a category with respect to a class of morphisms. But in general one has little control over the morphisms in the corresponding homotopy category. An excellent way to overcome this difficulty is Quillen's theory of model structures, which are structures consisting of three classes of morphisms (called {\it cofibrations}, {\it fibrations} and {\it weak equivalences}) satisfying five axioms \cite{Quillen67}. There are two advantages of a model structure: one is that we can realize the homotopy category as the subfactor category of cofibrant-fibrant objects; the other one is that the model structure may induce a $\mathbb{Z}$-linear structure and a pre-triangulated structure on the corresponding homotopy category \cite{Quillen67, Hovey99}.

However, the work of Happel \cite{Happel88} and the recent work of Iyama and Yoshino \cite{Iyama-Yoshino} show that for some additive categories with suspensions (in the sense of Heller \cite{Heller68}) and a class of morphisms determined by some homologically finite subcategories, one can do the same homotopy theory independent of model structures. Even we can model their work, but it seems superfluous. In this paper, we provide a framework for general additive categories with suspensions to construct homotopy theory based on the work in \cite{ZWLi2}. There we find the partial one-sided triangulated structures in additive categories with suspensions and show that they are proper settings for the construction of one-sided triangulated categories. Our framework covers many instances in algebra and geometry where homotopy theories are constructed. Compared with Quillen's theory, it seems more flexible and there is no obvious, at least to the author, a model structure for our homotopy theory.

We now give some details about our results. In an arbitrary additive category $\A$, we introduce the notion of a {\it localization triple} which consists of three full additive subcategories $\C,\X$ and $\F$ of $\A$ satisfying some conditions such that we can do localization for $\A$ with respect to a class $\mathcal{S}$ of morphisms of $\A$ decided by them. We define the homotopy category $\Ho(\C, \X, \F)$ of a localization triple $(\C,\X, \F)$ to be the corresponding localization $\A[\mathcal{S}^{-1}]$ which is equivalent to the subfactor category $(\C\cap \F)/\X$. For the aim of developing homotopy theory for a localization triple, we introduce the notion of a {\it pre-partial triangulated category} which is an additive category $\A$ with a partial left triangulated structure $(\Omega, \rmL(\F), \X)$ and a partial right triangulated structure $(\Sigma, \R(\C), \X)$ satisfying some compatible conditions. Our first main result is the following:
\vskip5pt
\noindent{\bf Theorem} (\ref{thm:pretricat} ) {\it  If $(\A, \Omega, \Sigma, \rmL(\F), \R(\C), \X)$ is a pre-partial triangulated category, then the subfactor category $(\C\cap \F)/\X$ is a pre-triangulated category.}
\vskip5pt
\noindent Special cases of the above theorem contain \cite[Theorem 2.2]{Jorgensen} and \cite[Corollary 4.10]{Beligiannis01}. When $(\C,\X,\F)$ is a localization triple in a pre-partial triangulated category $(\A, \Omega, \Sigma, \rmL(\F), \R(\C), \X)$, then the homotopy category $\Ho(\C,\X,\F)$ is a pre-triangulated category.

The concept of an exact model category was introduced in \cite{Hovey02, Gillespie13}. In brief, it is a weakly idempotent complete additive category which has both an exact structure and a model structure and the two structures are compatible. Exact model categories play more and more important roles in the study of algebraic triangulated categories \cite{Hovey02, Beligiannis01}, singularity categories \cite{Becker12}, relative homological algebras \cite{Beligiannis/Reiten07}, approximation theory \cite{Saorin/Stovicek} and algebraic geometry \cite{Gillespie12, Gillespie13}. For more background and overview about exact model categories we refer the reader to \cite{Hovey07, Stovicek}. For an exact category, the notion of a {\it Hovey triple} \cite{Gillespie12} plays the central role in the study of exact model category. In fact, a Hovey triple determines uniquely an exact model structure in a {\it weakly idempotent complete} (i.e. every split monomorphism is an inflation ) exact category \cite{Hovey02}. But for an arbitrary exact category, this is not necessarily the case \cite{Gillespie12}. Using our theory, we can develop the homotopy theory for a Hovey triple in an arbitrary exact category not depending on model structures, which is our second main result:

\vskip5pt
\noindent{\bf Theorem} (\ref{thm:homoHoveytriple}) {\it  Let $(\C, \mathcal{W}, \F)$ be a Hovey triple in an exact category $(\A, \E)$. Denote by $\X=\C\cap \mathcal{W}\cap \F$. Then

$(\mathrm{i})$ \ $(\C, \X, \F)$ is a localization triple.

$(\mathrm{ii})$ \ $(\A, 0, 0,, \rmL(\F), \R(\C), \X)$ is a pre-partial triangulated category. }

\vskip5pt

In general, the homotopy category of a pointed model category is a pre-triangulated category by Quillen's original work \cite{Quillen67}, see also \cite{Hovey99}. Since an exact model category is certainly pointed, its homotopy category carries a pre-triangulated structure induced by the exact model structure. Let $\M$ be an exact model structure induced by a Hovey triple $(\C,\mathcal{W},\F)$ in a weakly idempotent complete category $\A$. It is well known that the homotopy category $\Ho(\A)$ of $\A$ is equivalent to the subfactor category $(\C\cap \F)/\X$ (here $\X=\C\cap \mathcal{W}\cap \F$) \cite{Beligiannis/Reiten07, Gillespie11, Becker12}. By our second main theorem, the latter also has a pre-triangulated structure induced by the pre-partial triangulated structure. Our third main result shows that these two pre-triangulated structures are compatible.

\vskip5pt

\noindent{\bf Theorem} (\ref{thm:model vs triple}) {\it Let $\M$ be the exact model structure induced by the Hovey triple $(\C,\mathcal{W},\F)$ in a weakly idempotent complete exact category $(\A,\E)$. Denote by $\X=\C\cap\mathcal{W}\cap \F$, then the homotopy category $\Ho(\M)$ is equivalent to the subfactor category $(\C\cap \F)/\X$ as pre-triangulated categories.}

\vskip5pt

Throughout this paper, unless otherwise stated, that all subcategories of additive categories considered are full, closed under isomorphisms, all functors between additive categories are assumed to be additive.

\subsection*{Acknowledgements} I would like to thank Henning Krause, Xiao-Wu Chen, Yu Ye for
their helpful discussions and suggestions.

\section{Preliminaries}

In this section we recall some basic facts and notions on partial one-sided triangulated categories in \cite{ZWLi2}.

\subsection*{Stable categories of additive categories} Let $\C$ be an additive category and $\X$ an additive subcategory of $\C$. We denote by $\C/\X$ the {\it stable} or {\it factor category} of $\C$ modulo $\X$. Recall that the objects of $\C/\X$ are the objects of $\C$, and for two objects $A$ and $B$ the Home space $\Hom_{\C/\X}(A,B)$ is the quotient $\Hom_\C(A,B)/\X(A,B)$, where $\X(A,B)$ is the subgroup of $\Hom_\C(A,B)$ consisting of those morphisms factorizing through an object in $\X$. Note that the stable category $\C/\X$ is an additive category and the canonical functor $\C\to \C/\X$ is an additive functor. For a morphism $f\colon A\to B$ in $\C$, we use $\ul{f}$ to denote its image in $\C/\X$.

\subsection*{Right triangulated categories} If $\H$ is an arbitrary category endowed with a functor $\Sigma\colon \H\to \H$ (such a category is called a {\it category with suspension} \cite{Heller68}), then a {\it right $\Sigma$-sequence} in $\H$ is a sequence of the form $A\stackrel{f}\to B\stackrel{g}\to C \stackrel{h}\to \Sigma(A)$.
\begin{definition} \label{def:rtricat} (\cite[Definition 1.1]{ABM})\ Let $\T$ be an additive category endowed with an additive endofunctor $\Sigma$. Let $\Delta$ be a class of right $\Sigma$-sequences called {\it right triangles}. The triple $(\T, \Sigma, \Delta)$ is called a {\it right triangulated category} if $\Delta$ is closed under isomorphisms and the following four axioms hold:

	(RT1)  For any morphism $f\colon A\to B$, there is a right $\Sigma$-sequence $A\xrto{f} B\to C\to \Sigma(A)$ in $\Delta$. For any object $A\in \T$, the right $\Sigma$-sequence $0\to A\xrto{1_A} A\to 0$ is in $\Delta$.
	
	(RT2)  If $A\xrto{f} B\xrto{g} C\xrto{h} \Sigma(A)$ is a right triangle, so is $B\xrto{g} C\xrto{h} \Sigma(A) \xrto{-\Sigma(f)} \Sigma(B)$.

	(RT3)  If the rows of the following diagram are right triangles and the leftmost square is commutative, then there is a morphism $\gamma\colon C\to C'$ making the whole diagram commutative:	
\[\xy\xymatrixcolsep{2pc}\xymatrix@C16pt@R16pt{A\ar[r]^f\ar[d]_{\alpha}&B \ar[r]^g\ar[d]^\beta&C\ar[r]^-h\ar@{.>}[d]^\gamma & \Sigma(A)\ar[d]^{\Sigma(\alpha)}\\
A'\ar[r]^{f'}&B'\ar[r]^{g'}&C'\ar[r]^-{h'}&\Sigma(A')}
	\endxy\]

(RT4)  For any three right triangles: $A\xrto{f} B\xrto{l} C'\to \Sigma(A)$, $B\xrto{g} C\xrto{h} A'\xrto{j} \Sigma(B)$ and $A\xrto{g\circ f} C\to B'\to \Sigma{A}$, there is a commutative diagram	
\[\xy\xymatrixcolsep{2pc}\xymatrix@C16pt@R16pt{
A\ar[r]^f\ar@{=}[d]&B\ar[r]^l\ar[d]^g&C'\ar[r]\ar@{.>}[d]& \Sigma(A)\ar@{=}[d]\\
A\ar[r]^{g\circ f}&C\ar[r]\ar[d]^{h}&B'\ar[r]\ar@{.>}[d]&\Sigma(A)\ar[d]^{\Sigma(f)}\\
&A'\ar@{=}[r]\ar[d]^{j}&A'\ar[r]^-{j}\ar[d]&\Sigma(B)\\
&\Sigma(B)\ar[r]^{\Sigma(l)}&\Sigma(C')}
	\endxy\]
	such that the second column from the right is a right triangle.
\end{definition}
When the endofunctor $\Sigma$ is an auto-equivalence, a right triangulated category $(\T, \Sigma, \Delta)$ is a triangulated category in the sense of \cite{Verdier}. The notion of a {\it left triangulated category} is defined dually.

\subsection*{Partial right triangulated categories}
Let $\A$ be an additive category endowed with an additive endofunctor $\Sigma\colon \A\to \A$. If $\X$ is an additive subcategory of $\A$, then a morphism $f\colon A\to B$ in $\A$ is said to be an {\it $\X$-monic} if the induced morphism $f^*=\Hom_\A(f, \X)\colon \Hom_\A(B, \X)\to \Hom_\A(A, \X)$ is surjective. The notion of an {\it$\X$-epic} is defined dually. Recall that a morphism $f\colon A\to X$ in $\A$ is called an {\it $\X$-preenvelope} (also called a {\it left $\X$-approximation} of $A$ in the literature) if $f$ is an $\X$-monic and $X\in \X$. Dually a morphism $g\colon X\to A$ is called an {\it $\X$-precover} if $g$ is an $\X$-epic and $X\in \X$. If $\C$ is an additive subcategory of $\A$, then a right $\Sigma$-sequence $A\stackrel{f}\to B\stackrel{g}\to C\stackrel{h}\to \Sigma(A)$
in $\A$ is called a {\it right $\C$-sequence} if $C\in \C$, $g$ is a {\it weak cokernel} of $f$ (i.e. the induced sequence
$\Hom_\A(C, \A)\to \Hom_\A(B, \A)\to \Hom_\A(A, \A)$ is exact), and $h$ is a weak cokernel of $g$.

\begin{definition}\label{def:prtc} (\cite[Definition 2.2]{ZWLi2}) \ Let $\A$ be an additive category endowed with an additive endofunctor $\Sigma$. Let $\X\subseteq \C$ be two additive subcategories of $\A$ and $\R(\C)$ a class  of right $\C$-sequences (called {\it right $\C$-triangles}).  The quadruple $(\A, \Sigma, \R(\C), \X)$ is said to be a {\it partial right triangulated category} if $\R(\C)$ is closed under isomorphisms and finite direct sums and the following axioms hold:

(PRT1) (i) For each $A\in \C$, there is a right $\C$-triangle $A\xrto{i} X\to U\to \Sigma(A)$ with $i$ an $\X$-preenvelope.

 (ii)  For each morphism $f\colon A\to B$ in $\C$, $A\xrightarrow{\left(\begin{smallmatrix}
	1 \\
	f
	\end{smallmatrix}\right)} A\oplus B \xrightarrow{(f, -1)} B\xrto{0} \Sigma(A)$ is in $\R(\C)$.
\vskip5pt
(iii) If $A\xrto{i} X\to U\to \Sigma(A)$ is in $\R(\C)$ with $i$ an $\X$-preenvelope in $\C$, then for any morphism $f\colon  A\to B$ in $\C$, there is a right $\C$-triangle $A\xrightarrow{\left(\begin{smallmatrix}
	i \\
	f
	\end{smallmatrix}\right)} X\oplus B \to N\to \Sigma(A)$.

(PRT2)  For any commutative diagram of right $\C$-triangles
\[\xy\xymatrixcolsep{2pc}\xymatrix@C16pt@R16pt{
A\ar[r]^f\ar[d]_{\alpha}&B\ar[r]\ar[d]&C\ar[r]\ar[d]^\gamma& \Sigma(A)\ar[d]^{\Sigma(\alpha)}\\
A'\ar[r]&X\ar[r]^s&U\ar[r]&\Sigma(A')}
\endxy\]
	with $X\in \X$, if $\alpha$ factors through $f$, then $\gamma$ factors through $s$.

(PRT3)  If the rows of the following diagram are in $\R(\C)$ with the leftmost square commutative, then there is a morphism $\gamma\colon C\to C'$ making the whole diagram commutative:
		\[\xy\xymatrixcolsep{2pc}\xymatrix@C16pt@R16pt{A\ar[r]^f\ar[d]_{\alpha}&B \ar[r]^g\ar[d]^\beta&C\ar[r]^-h\ar@{.>}[d]^\gamma & \Sigma(A)\ar[d]^{\Sigma(\alpha)}\\
A'\ar[r]^{f'}&B'\ar[r]^{g'}&C'\ar[r]^-{h'}&\Sigma(A')}
	\endxy\]

(PRT4)  If $A\xrto{f} B\xrto{l} C'\to \Sigma(A)$, $B\xrto{g} C \xrto{h} A'\xrto{j} \Sigma(B)$ and $A\xrightarrow{g\circ f} C \to B'\to \Sigma(A)$ are in $\R(\C)$ such that $f$ and $g$ are $\X$-monics in $\C$, then there is a commutative diagram
\[\xy\xymatrixcolsep{2pc}\xymatrix@C16pt@R16pt{
A\ar[r]^f\ar@{=}[d]&B\ar[r]^l\ar[d]^g&C'\ar[r]\ar@{.>}[d]^r& \Sigma(A)\ar@{=}[d]\\
A\ar[r]^{g\circ f}&C\ar[r]\ar[d]^{h}&B'\ar[r]\ar@{.>}[d]&\Sigma(A)\ar[d]^{\Sigma(f)}\\
&A'\ar@{=}[r]\ar[d]^{j}&A'\ar[r]^-{j}\ar[d]&\Sigma(B)\\
&\Sigma(B)\ar[r]^{\Sigma(l)}&\Sigma(C')}
	\endxy\]
such that the second column from the right is in $\R(\C)$ with $r$ an $\X$-monic.
\end{definition}

The notion of a {\it partial left triangulated category} is defined dually.



\subsection*{The triangulation of the subfactors of partial one-sided triangulated categories}
 Let $(\A,\Sigma, \R(\C), \X)$ be a partial right triangulated category. For each $A\in \C$, we {\it fix} a right $\C$-triangle $A\xrto{i^A} X^A\xrto{p^A} U^A\xrto{q^A} \Sigma(A)$ with $i^A$ an $\X$-preenvelope. Then there is a functor $\Sigma^\X\colon \C/\X\to \C/\X$ by sending each object $A$ to $U^A$ and each morphism $\ul{f}\colon A\to B$ to $\ul{\kappa}^f$, where the morphism $\kappa^f$ is defined by the following commutative diagram:
\begin{equation}\label{kappaf}
\xy\xymatrixcolsep{2pc}\xymatrix@C16pt@R16pt{
A\ar[r]^{i^A}\ar[d]_f&X^A\ar[r]^{p^A}\ar[d]^{\sigma^f}&U^A\ar[r]^-{q^A}\ar[d]^{\kappa^f}& \Sigma(A)\ar[d]^{\Sigma(f)}\\
B\ar[r]^{i^B}&X^B\ar[r]^{p^B}&U^B\ar[r]^-{q^B}&\Sigma(B)}
\endxy
\end{equation}
A {\it standard right triangle} in the subfactor category $\C/\X$ is an induced right $\Sigma^\X$-sequence $A\stackrel{\ul{f}}\to B\stackrel{\ul{g}}\to C\xrightarrow{\ul{\xi}(f,g)}\Sigma^\X(A)$ by a right $\C$-triangle$A\stackrel{f}\to B\stackrel{g}\to C\stackrel{h}\to \Sigma(A) $ with $f$ an $\X$-monic in $\C$ and the following commutative diagram
\begin{equation}\label{xif}
\xy\xymatrixcolsep{2pc}\xymatrix@C16pt@R16pt{
A\ar[r]^f\ar@{=}[d]&B\ar[r]^g\ar[d]^{\delta^f}&C\ar[r]^-h\ar[d]^{\xi(f,g)}& \Sigma(A)\ar@{=}[d]\\
A\ar[r]^{i^A}&X^A\ar[r]^{p^A}&U^A\ar[r]^-{q^A}&\Sigma(A)}
\endxy
\end{equation}
 Denote by $\Delta^\X$ the class of right $\Sigma^\X$-sequences in $\C/\X$ which are isomorphic to standard right triangles. The right triangles in $\Delta^\X$ are called {\it distinguished right triangles}.

  Dually, if $(\A, \Omega, \rmL(\C), \X)$ is a partial left triangulated category, for each object $A\in \C$, we {\it fix} a left $\C$-triangle $\Omega(A)\xrto{\nu_A} U_A\xrto{\iota_A} X_A\xrto{\pi_A} A$ with $\pi_A$ an $\X$-precover, then we can define an additive endofunctor $\Omega_\X\colon \C/\X\to \C/\X$, and the {\it standard left triangles} in $\C/\X$. Denote by $\nabla_\X$ the class of left $\Omega_\X$-sequences in $\C/\X$ which are isomorphic to standard left triangles. Then we have:

 \begin{theorem} \label{thm:main} $($\cite[Theorem 3.2]{ZWLi2}$)$  \ $(\mathrm{i})$ \ If $(\A, \Sigma, \R(\C), \X)$ is a partial right triangulated category, then $(\C/\X, \Sigma^\X, \Delta^\X)$ is a right triangulated category.
	
	$(\mathrm{ii})$ \ If $(\A, \Omega, \rmL(\C), \X)$ is a partial left triangulated category, then $(\C/\X, \Sigma_\X, \nabla_\X)$ is a left triangulated category.
\end{theorem}

\section{The homotopy categories of additive categories}

 In this section, we define the homotopy category of an additive category based on the notion of a localization triple.

\subsection*{The localization of categories}
\begin{definition} (\cite{Gabriel/Zisman67}) \ \label{def:localization} Let $\A$ be a category and let $\mathcal{S}$ be a class of morphisms of $\A$. The {\it localization of $\A$ with respect to $\mathcal{S}$} means a category $\A[\mathcal{S}^{-1}]$ together with a functor $\gamma \colon \A\to \A[\mathcal{S}^{-1}]$ such that
\vskip5pt
(i)  $\gamma (s)$ is an isomorphism for each $s\in \mathcal{S}$, and

(ii)  whenever $F\colon \A\to \mathcal{B} $ is a functor carrying elements of $\mathcal{S}$ to isomorphisms, there exists a unique functor $F'\colon \A[\mathcal{S}^{-1}]\to \mathcal{B}$ such that $F'\circ \gamma=F$.
\end{definition}

 \noindent The second condition of Definition \ref{def:localization} shows that any two localizations of $\A$ with respect to $\mathcal{S}$ are canonically isomorphic. A general construction of the category $\A[\mathcal{S}^{-1}]$ is given by Gabriel and Zisman \cite{Gabriel/Zisman67}, but there is a foundational set-theoretic obstruction to its existence.

\subsection*{The homotopy category of an additive category}

Let $\A$ be an additive category. Assume $\C, \D$ and $\X$ are three additive subcategories of $\A$ such that $\X\subseteq \C\cap \D$. We denote by $$\C^{\perp_{\A/\X}}=\{W\in \A \ | \ \Hom_{\A/\X}(\C, W)=0\} \ \mbox{and} \ {^{\perp_{\A/\X}}\D}=\{W\in \A \ | \ \Hom_{\A/\X}(W, \D)=0\}.$$

\begin{definition} \ \label{htriple}The triple $(\C, \X, \D)$ is called a {\it localization triple} of $\A$ if

(a) For each $A\in \A$, there is a fixed sequence $W_A\xrto{\omega_A} Q(A)\xrto{r_A}A$ such that $r_A$ is a $\C$-precover with $\omega_A$ a weak kernel of $r_A$ and $W_A\in \C^{\perp_{\A/\X}}$.

(b) For each $A\in \A$, there is a fixed sequence $A\xrto{j^A} R(A)\xrto{\tau^A} W^A$ such that $i^A$ is a $\D$-preenvelope with $\tau^A$ a weak cokernel of $j^A$ and $W^A\in {^{\perp_{\A/\X}}\D}$.

(c) If $A\in \D$, then $Q(A)\in \C\cap \D$, and if $A\in \C$, then $R(A)\in \C\cap \D$.
\end{definition}

Now we fix a localization triple $(\C, \X, \D)$ in an additive category $\A$.

\begin{lemma} \ \label{lem:creplacement} For any morphism $f\colon A\to B$ in $\A$, there exists a morphism $\check{f}\colon Q(A)\to Q(B)$ such that
the following diagram commutes
\[\xy\xymatrixcolsep{2pc}\xymatrix@C16pt@R16pt{
Q(A)\ar[r]^{\check{f}}\ar[d]_{r_A}&Q(B)\ar[d]^{r_B}\\
A\ar[r]^{f}&B}
\endxy\]
and $\ul{\check{f}}$ is uniquely determined by $\ul{f}$ in the stable category $\A/\X$.
\end{lemma}

\begin{proof} \ Apply $\Hom_\A(Q(A),-)$ to the fixed sequence $W_B\xrto{\omega_B} Q(B)\xrto{r_B} B$, we obtain a right exact sequence
\begin{equation*}
\Hom_\A(Q(A), W_B)\xrto{(\omega_{B})_*} \Hom_\C(Q(A), Q(B))\xrto{(r_{B})_*} \Hom_{\A}(Q(A), B)\to 0.
\end{equation*}
Then for $f\circ r_A\in \Hom_\A(Q(A), B)$, there exists a morphism $\check{f}\colon Q(A)\to Q(B)$ such that $f\circ r_A=r_B\circ \check{f}$. 
If there is another morphism $g\colon A\to B$ such that $\ul{g}=\ul{f}$ in $\A/\X$, then we can factor $f-g$ as $A\xrto{t} X\xrto{s} B $ with some $X\in \X$. Since $f\circ r_A=r_B\circ \check{f}$ and $g\circ r_A=r_B\circ \check{g}$, we have $(f-g)\circ r_A=r_B\circ (\check{f}-\check{g})$. Since $X\in \X\subseteq \C$, there is a morphism $h\colon X\to Q(B)$ such that $s=r_B\circ h$ since $r_B$ is a $\C$-epic. Then $r_B\circ h\circ t\circ r_A=s\circ t\circ r_A=(f-g)\circ r_A=r_B\circ (\check{f}-\check{g})$. Thus there is a morphism $l\colon Q(A)\to W_B$ such that $(\check{f}-\check{g})-h\circ t\circ r_A=\omega_B\circ l$ since $\omega_B$ is a weak kernel of $r_B$. Therefore $\ul{\check{f}}-\ul{\check{g}}=\ul{\omega}_B\circ \ul{l}+\ul{h}\circ \ul{t}\circ \ul{r}_A=0$ in $\C/\X$.\end{proof}

Dually, we have
\begin{lemma} \ \label{lem:freplacement} For any morphism $f\colon A\to B$ in $\A$, there exists a morphism $\hat{f}\colon R(A)\to R(B)$ such that the following diagram commutes
\[\xy\xymatrixcolsep{2pc}\xymatrix@C16pt@R16pt{
A\ar[r]^f\ar[d]_{j^A}&B\ar[d]^{j^B}\\
R(A)\ar[r]^{\hat{f}}&R(B)}
\endxy\]
and $\ul{\hat{f}}$ is uniquely determined by $\ul{f}$ in the stable category $\A/\X$.
\end{lemma}

\begin{remark} \label{functor} \  The uniqueness statement in Lemma \ref{lem:creplacement} implies that we can define a functor $Q\colon \A/\X\to\C/\X$ by sending $A$ to $Q(A)$ and $\ul{f}\colon A\to B$ to $\ul{\check{f}}$, which is a right adjoint of the inclusion $\C/\X\hookrightarrow \A/\X$. Similarly, Lemma \ref{lem:freplacement} implies that we can define a functor $R\colon \A/\X\to \D/\X$ by sending $A$ to $R(A)$ and $\ul{f}$ to $\ul{\hat{f}}$, which is a left adjoint of the inclusion $\D/\X\hookrightarrow \A/\X$.
\end{remark}

Let $\mathrm{Mor}(\A)$ be the class of morphisms in $\A$. Then we define
\begin{equation} \label{we}
\mathcal{S}=\{s\in \mathrm{Mor}(\A) \ | \ RQ(\ul{s}) \ \mbox{is an isomorphism in } \ (\C\cap \D)/\X\}.
 \end{equation}
  For each object $A\in \A$, the morphism $r_A\colon Q(A)\to A$ is in $\mathcal{S}$ and $j^{A}\colon A\to R(A)$ is in $\mathcal{S}$ if $A\in \C$. By Lemmas \ref{lem:creplacement}-\ref{lem:freplacement}, $\mathcal{S}$ satisfies two out of three property.

\begin{lemma}  \label{lefthomotopy} \ Let $\mathcal{B}$ be a category, and let $F\colon \A\to \mathcal{B}$ be a functor which takes elements of $\mathcal{S} $ to isomorphisms. If $f, g\colon A\to B$ are morphisms in $\A$ such that $\ul{f}=\ul{g}$ in $\A/\X$, then $F(f)=F(g)$.
\end{lemma}
\begin{proof} \ Since $\ul{f}=\ul{g}$ in $\A/\X$, there is a factorization of $f-g$: $A\stackrel{v}\to X\stackrel{u}\to B$ for some $X\in \X$. Since $\left(\begin{smallmatrix}
1_B \\
0
\end{smallmatrix}\right)\colon B\to B\oplus X$ is in $\mathcal{S}$, we know that $F\left(\begin{smallmatrix}
1_B \\
0
\end{smallmatrix}\right)$ is an isomorphism by assumption. Thus $F(1_B, u)=F(1_B, 0)$ since $F(1_B, u)\circ F\left(\begin{smallmatrix}
1_B \\
0
\end{smallmatrix}\right)=F(1_B)=F(1_B,0)\circ F\left(\begin{smallmatrix}
1_B \\
0
\end{smallmatrix}\right)$. Then we have $F(f)=F(1_B, u)\circ F\left(\begin{smallmatrix}
g\\
v
\end{smallmatrix}\right)=F(1_B, 0)\circ F\left(\begin{smallmatrix}
g\\
v
\end{smallmatrix}\right)=F(g)$.
\end{proof}

\begin{theorem} \label{thm:localization} \ Let $(\C, \X, \D)$ be a localization triple in an additive category $\A$. Then the localization $\gamma\colon \A\to \A[\mathcal{S}^{-1}]$ of $\A$ with respect to $\mathcal{S}$ as defined in $(\ref{we})$ exists.
\end{theorem}
\begin{proof} \  We define a category $\A[\mathcal{S}^{-1}]$ as follows, it has the same objects as $\A$ and for $A, B\in \A$,  $$\Hom_{\A[\mathcal{S}^{-1}]}(A, B)=\Hom_{(\C\cap \D)/\X}(R\circ Q(A), R\circ Q(B))$$
where $R, Q$ are defined in Remark \ref{functor}.
Then there is a functor $\gamma\colon \A\to \A[\mathcal{S}^{-1}]$ which is identity on objects and sends a morphism $f\colon A\to B$ to $R\circ Q(\ul{f}) =\ul{\hat{\check{f}}}\colon R(Q(B))\to R(Q(A))$. By Remark \ref{functor}, $\gamma$ is a functor which sends elements of $\mathcal{S}$ defined as in (\ref{we}) to isomorphisms. Let $F\colon \A\to \mathcal{B}$ be a functor that takes elements of $\mathcal{S}$ to isomorphisms in $\mathcal{B}$. We will show that there is a unique functor $F'\colon \A[\mathcal{S}^{-1}]\to \mathcal{B}$ such that $F'\circ \gamma=F$ and then $\gamma\colon \A\to \A[\mathcal{S}^{-1}]$ is a localization $\A[\mathcal{S}^{-1}]$ of $\A$ with respect to $\mathcal{S}$.

 For every object $A$ in $\A[\mathcal{S}^{-1}]$, we let $F'(A)=F(A)$. If $f\colon A\to B$ is a morphism in $\A[\mathcal{S}^{-1}]$, then $f$ can be extended to a morphism $\gamma(j^{Q(B)})\circ \gamma(r_B)^{-1}\circ f \circ \gamma(r_A)\circ \gamma(j^{Q(A)})^{-1}\colon R(Q(A))\to R(Q(B))$ in $\A[\mathcal{S}^{-1}]$. Thus there is some morphism $f'\colon R(Q(B))\to R(Q(A))$ in $\C\cap \D$ such that $\gamma(j^{Q(B)})\circ \gamma(r_B)^{-1}\circ f \circ \gamma(r_A)\circ \gamma(j^{Q(A)})^{-1}=\ul{f}'=\gamma(f')$. Then $f=\gamma(r_{B})\circ \gamma(j^{Q(B)})^{-1}\circ \gamma(f') \circ \gamma(j^{Q(A)})\circ \gamma(r_{A})^{-1}$. By Lemma \ref{lefthomotopy}, $F(f')$ only depends on $\ul{f}'$, and therefore only on $f$, so we can define
 $$F'(f)=F(r_B)\circ F(j^{Q(B)})^{-1}\circ F(f')\circ F(j^{Q(A)})\circ F(r_A)^{-1}.$$
 If $f\colon A\to A$ is the identity morphism in $\A[\mathcal{S}^{-1}]$, then we can take $f'=1_{R(Q(A))}$ by Lemmas \ref{lem:creplacement}-\ref{lem:freplacement} and thus $F'(f)=1_{F'(A)}$. Given two composable morphisms $f\colon A\to B$ and $g\colon B\to C$ in $\A[\mathcal{S}^{-1}]$, by the above proof, there are morphisms $f'\colon R(Q(A))\to R(Q(B))$ and $g'\colon R(Q(B))\to R(Q(C))$ in $\C\cap \D$ such that $f=\gamma(r_{B})\circ \gamma(j^{Q(B)})^{-1}\circ \gamma(f') \circ \gamma(j^{Q(A)})\circ \gamma(r_{A})^{-1}$ and $g=\gamma(r_{C})\circ \gamma(j^{Q(C)})^{-1}\circ \gamma(g') \circ \gamma(j^{Q(B)})\circ \gamma(r_{B})^{-1}$. Then $g\circ f$ can be represented as $\gamma(r_{C})\circ \gamma(j^{Q(C)})^{-1}\circ \gamma(g'\circ f') \circ \gamma(j^{Q(A)})\circ \gamma(r_{A})^{-1}$. Then
 \begin{align*}
 F'(g\circ f)&=F(r_C)\circ F(j^{Q(C)})^{-1}\circ F(g'\circ f')\circ F(j^{Q(A)})\circ F(r_A)^{-1}\\
 &=[F(r_C)\circ F(j^{Q(C)})^{-1}\circ F(g')\circ F(j^{Q(B)})\circ F(r_B)^{-1}] \\
 &\circ [F(r_B)\circ F(j^{Q(B)})^{-1}\circ F(f')\circ F(r^{Q(A)})\circ F(r_A)^{-1}]\\
 &=F'(g)\circ F'(f).
 \end{align*}
 This shows that $F'$ is a functor.

 To see that $F'\circ \gamma=F$, we note that $\gamma$ is the identity on objects, and $F'$ was defined to agree with $F$ on objects. If $f\colon A\to B$ is a morphism in $\A$, by Lemmas \ref{lem:creplacement}-\ref{lem:freplacement} we have a commutative diagram
\[\xy\xymatrixcolsep{2pc}\xymatrix@C16pt@R16pt{
 R(Q(A)) \ar[r]^{\hat{\check{f}}}& R(Q(B))\\
Q(A)\ar[r]^{\check{f}}\ar[d]_{r_A}\ar[u]^{j^{Q(A)}}&Q(B)\ar[d]^{r_B}\ar[u]_{j^{Q(B)}}\\
A\ar[r]^f&B}
\endxy\]
  Since $F$ takes elements of $\mathcal{S}$ to isomorphisms in $\mathcal{B}$, we have
 $$F(f)=F(r_B)\circ F(j^{Q(B)})^{-1}\circ F(\hat{\check{f}})\circ F(j^{Q(A)})\circ F(r_A)^{-1}$$
 which implies $F'\circ \gamma(f)=F(f)$ since $\gamma(f)=\ul{\hat{\check{f}}}$ .

 Since every morphism $f\colon A\to B$ in $\A[\mathcal{S}^{-1}]$ can be represented as a composite $\gamma(r_{B})\circ \gamma(j^{Q(B)})^{-1}\circ \gamma(f') \circ \gamma(j^{Q(A)})\circ \gamma(r_{A})^{-1}$, we know that $F'$ is the unique functor satisfying $F'\circ \gamma=F$.\end{proof}

 The category $\A[\mathcal{S}^{-1}]$ in Theorem \ref{thm:localization} will be called the {\it homotopy category} of the localization triple $(\C, \X, \D)$ and we will use $\Ho(\C,\X,\D)$ to denote it.

\begin{remark} \label{remak:triangleequi}  (i)  Since for any object $A$ in $\Ho(\C,\X,\D)$, $\gamma(j^{Q(A)})\circ \gamma(r_A)^{-1}\colon A\to R(Q(A))$ is an isomorphism and the embedding $E\colon (\C\cap\D)/\X\hookrightarrow \Ho(\C,\X,\D)$ is fully faithful, we know that the embedding $E$ is an equivalence. In particular, $\Ho(\C,\X,\D)$ is an additive category.

(ii) One reason for us not using $(\C\cap \F)/\X$ as the homotopy category of a homotopy triple is since $\C\cap \F$ is inadequate, because many constructions in homotopy theory create objects that may not be in $\C\cap \F$.

\end{remark}

\begin{example}   (i) Let $\A$ be an additive category and $\X$ an additive subcategory of $\A$. Then $(\A, \X, \A)$ is a localization triple and $\mathcal{S}$ is the class of {\it stable equivalences}. Here a morphism $f\colon A\to B$ is said to be a stable equivalence if its image $\ul{f}$ is an isomorphism in $\A/\X$, i.e. there is a morphism $g\colon B\to A$ such that both $g\circ f-1_A$ and $f\circ g-1_B$ factor through some object in $\X$. The corresponding homotopy category is the stable category $\A/\X$.

(ii) Let $R$ be a ring. Take $\A$ to be the category of chain complexes of right $R$-modules, $\C$ the subcategory of DG-projective complexes and $\X$ the subcategory of acyclic DG-projective complexes, then $(\C, \X, \A)$ is a localization triple. In this case, the class $\mathcal{S}$ as defined in (\ref{we}) is just the class of quasi-isomorphisms, thus the corresponding homotopy category is the derived category $\rmD(R)$ of right $R$-modules.
\end{example}

\section{Stabilizing subcategories of partial one-sided triangulated categories}

In this section we introduce the notion of stabilizing subcategories for partial one-sided triangulated categories and show that the stable categories of stabilizing subcategories are one-sided triangulated categories.

\subsection*{Stabilizing subcategories}

Let $(\A, \Sigma, \R(\C), \X)$ be a partial right triangulated category. Assume that $\G$ is an additive subcategory of $\A$ such that $\X\subseteq \G\subseteq \C$. Then $\G$ is said to be {\it special preenveloping} in $\C$ if for each $A\in \C$, there is a right $\C$-triangle $A\xrto{j^A} R(A)\to W^A\to \Sigma(A)$
in $\R(\C)$ such that $j^A$ is a $\G$-preenvelope and $W^A\in {^{\perp_{\C/\X}}\G}$ (we fix one).

If $\G$ is a special preenveloping subcategory of $\C$, then by Remark \ref{functor}, the embedding $E^\C\colon \G/\X\hookrightarrow \C/\X$ has a left adjoint $R\colon \C/\X\to \G/\X$ by sending $A$ to $R(A)$ and $\ul{f}\colon A\to B$ to $\ul{\hat{f}}\colon R(A)\to R(B)$, where $\hat{f}$ is as defined in Lemma \ref{lem:freplacement}.

If $(\A, \Sigma, \R(\C), \X)$ is a partial right triangulated category, the subfactor category $\C/\X$ has a right triangulated structure $(\Sigma^\X, \Delta^\X)$ by Theorem \ref{thm:main}. For any morphism $f\colon A\to B$ in $\C$, there is a right triangle $A\xrto{\ul{f}} B\xrto{\ul{g}} C\xrto{\ul{h}}\Sigma^\X(A)$ in $\Delta^\X$.

\begin{definition} \ \label{stabilizing} Let $(\A, \Sigma, \R(\C),\X)$ be a partial right triangulated category. A special preenveloping subcategory $\G$ of $\C$ is said to be {\it stabilizing} if for any diagram of the form
 \[\xy\xymatrixcolsep{2pc}\xymatrix@C16pt@R16pt{
A\ar[r]^-{\ul{f}}\ar[d]_{\ul{j}^A}&B\ar[r]\ar[d]^{\ul{j}^B}&C\ar[r]\ar@{.>}[d]^{\ul{t}}& \Sigma^\X(A)\ar[d]^{\Sigma^\X(\ul{j}^A)}\\
R(A)\ar[r]^-{\ul{\hat{f}}}&R(B)\ar[r]&D\ar[r]&\Sigma^\X(R(A))}
\endxy\]
where the rows are right triangles in $\Delta^\X$, there is a morphism $\ul{t}\colon C\to D$ such that the whole diagram commutative and $R(\ul{t})$ is an isomorphism in $\G/\X$.
\end{definition}

\begin{remark}\label{rem:rstabilizing} The subcategory $\C$ is always stabilizing in $(\A, \Sigma, \R(\C), \X)$ since we can take $R(A)=A$ for each object $A\in \C$ and the existence of $t$ is guaranteed by \cite[Lemma 2.4 (ii)]{ZWLi2}.
\end{remark}

Now assume that $\G$ is stabilizing in the partial right triangulated category $(\A, \Sigma, \R(\C), \X)$. Let $G^\X=R\circ \Sigma^\X\circ E^\C$. We call the right $G^\X$-sequence $A\xrto{\ul{f}} B\xrto{R(\ul{g})} R(C)\xrto{R(\ul{h})} G^\X(A)$ in $\G/\X$ a {\it standard right triangle} if  $A\xrto{\ul{f}} B\xrto{\ul{g}} C\xrto{\ul{h}}\Sigma^\X(A)$ is a distinguished right triangle in $\Delta^\X$. We use $\Delta^\X(\G)$ to denote the class of right $G^\X$-sequences which are isomorphic to standard right triangles in $\G/\X$.

\begin{lemma} \label{distriangle} If $A\xrto{\ul{f}} B\xrto{\ul{g}} C\xrto{\ul{h}}\Sigma^\X(A)$ is in $\Delta^\X$, then $R(A)\xrto{R(\ul{f})} R(B)\xrto{R(\ul{g})} R(C)\xrto{G^\X(\ul{j}^A)\circ R(\ul{h})}G^\X(R(A))$ is in $\Delta^\X(\G)$.
\end{lemma}
\begin{proof} \ Assume that the corresponding distinguished right triangle of $R(\ul{f})=\ul{\hat{f}}\colon R(A)\to R(B)$ is $R(A)\xrto{\ul{\hat{f}}} R(B)\xrto{\ul{u}} D\xrto{\ul{v}}\Sigma^\X(A)$ in $\C/\X$. Since $\G$ is stabilizing, there is a morphism $\ul{t}\colon C\to D$ such that $R(\ul{t})$ is an isomorphism in $\G/\X$ and the following diagram of distinguished right triangles is commutative:
\[\xy\xymatrixcolsep{2pc}\xymatrix@C16pt@R16pt{
A\ar[r]^-{\ul{f}}\ar[d]_{\ul{j}^A}&B\ar[r]^{\ul{g}}\ar[d]^{\ul{j}^B}&C\ar[r]^-{\ul{h}}\ar[d]^{\ul{t}}& \Sigma^\X(A)\ar[d]^{\Sigma^\X(\ul{j}^A)}\\
R(A)\ar[r]^-{\ul{\hat{f}}}&R(B)\ar[r]^{\ul{u}}&D\ar[r]^-{\ul{v}}&\Sigma^\X(R(A))}
\endxy\]
Then we have a commutative diagram of $G^\X$-sequences in $\G/\X$
\[\xy\xymatrixcolsep{2pc}\xymatrix@C38pt@R16pt{
R(A)\ar[r]^-{R(\ul{f})}\ar@{=}[d]&R(B)\ar[r]^-{R(\ul{g})}\ar@{=}[d]&R(C)\ar[r]^-{G^\X(\ul{j}^A)\circ R(\ul{h})}\ar[d]^{R(\ul{t})}& G^\X(R(A))\ar@{=}[d]\\
R(A)\ar[r]^-{R(\ul{f})}&R(B)\ar[r]^-{R(\ul{u})}&R(D)\ar[r]^-{R(\ul{v})}&G^\X(R(A))}
\endxy\]
Since the second row is a standard right triangle in  $\G/\X$ and $R(\ul{t})$ is an isomorphism, we know that the first row is in $\Delta^\X(\G)$.
\end{proof}

\begin{proposition} \label{rstabilizing} Let $(\A, \Sigma, \R(\C),\X)$ be a partial right triangulated category with $\G$ a stabilizing subcategory. Then $(\G/\X, G^\X, \Delta^\X(\G))$ is a right triangulated category.
\end{proposition}

\begin{proof}  (i)  We verify (RT1)-(RT4) of Definition \ref{def:rtricat} one by one.
	
	(RT1)   For each object $A\in \G$, there is a standard right triangle $0\to A\stackrel{1}\to A\to G^\X(0)$ induced by the right triangle $0\to A\xrto{1} A\to \Sigma^\X(0)$ in $\Delta^\X$. By the definition of a standard right triangle in $\G/\X$, given any morphism $f \colon A\to B$ in $\G$, there is a standard right triangle $A\xrto{\ul{f}} B\to R(C)\to G^\X(A)$.
	
	(RT2)   Let $A\xrto{\ul{f}} B\xrto{R(\ul{g})} R(C)\xrto{R(\ul{h})} G^\X(A)$ be a standard right triangle induced by the distinguished right triangle $A\xrto{\ul{f}} B\xrto{\ul{g}} C\xrto{\ul{h}}\Sigma^\X(A)$ in $\Delta^\X$. Then $B\xrto{\ul{g}}C\xrto{\ul{h}}\Sigma^\X(A)\xrto{-\Sigma^\X(\ul{f})}\Sigma^\X(B)$ is a right triangle in $\Delta^\X$. Thus $B\xrto{R(\ul{g})}R(C)\xrto{R(\ul{h})}G^\X(A)\xrto{-G^\X(\ul{f})}G^\X(B)$ is in $\Delta^\X(\G)$ by Lemma \ref{distriangle}.

(RT3)   Assume that we have a diagram of standard right triangles in $\G/\X$
\[\xy\xymatrixcolsep{2pc}\xymatrix@C16pt@R16pt{
A\ar[r]^-{\ul{f}}\ar[d]_{\ul{\alpha}}&B\ar[r]^-{R(\ul{g})}\ar[d]^{\ul{\beta}}&R(C)\ar[r]^-{R(\ul{h})}& G^\X(A)\ar[d]^{G^\X(\ul{\alpha})}\\
A'\ar[r]^-{\ul{f}'}&B'\ar[r]^-{R(\ul{g}')}&R(C')\ar[r]^-{R(\ul{h}')}&G^\X(A')}
\endxy\]
with the leftmost square commutative. Then there is a morphism $t\colon C\to C'$ making the following diagram of right triangles commutative
		\[\xy\xymatrixcolsep{2pc}\xymatrix@C16pt@R16pt{
A\ar[r]^-{\ul{f}}\ar[d]_{\ul{\alpha}}&B\ar[r]^-{\ul{g}}\ar[d]^{\ul{\beta}}&C\ar[r]^-{\ul{h}}\ar[d]^{t}& \Sigma^\X(A)\ar[d]^{\Sigma^\X(\alpha)}\\
A'\ar[r]^-{\ul{f}'}&B'\ar[r]^-{\ul{g}'}&C'\ar[r]^-{\ul{h}'}&\Sigma^\X(A')}
\endxy\]	by the (RT3) axiom of $\Delta^\X$. Thus $R(\ul{t}): R(C)\to R(C')$ is the desired filler.

(RT4). Assume that we have three standard right triangles $A\xrto{\ul{f}}B\xrto{R(\ul{l})} R(C')\xrto{R(\ul{m})}G^\X(A)$, $B\xrto{R(\ul{g})} C\xrto{R(\ul{h})} R(A')\xrto{R(\ul{n})} G^\X(B)$ and $A\xrightarrow{\ul{g\circ f}} C\xrto{R(\ul{u})}R(B')\xrightarrow{\ul{v}}G^\X(A)$ in $\G/\X$. Then we have a commutative diagram of distinguished right triangles in $\C/\X$ by the (RT4) axiom of $\Delta^\X$:
	\[\xy\xymatrixcolsep{2pc}\xymatrix@C16pt@R16pt{
A\ar[r]^-{\ul{f}}\ar@{=}[d]&B\ar[r]^-{\ul{l}}\ar[d]^{\ul{g}}& C'\ar[r]^-{\ul{m}}\ar[d]^{\ul{r}}& \Sigma^\X(A)\ar@{=}[d]\\
A\ar[r]^-{\ul{g\circ f}}&C\ar[r]^-{\ul{u}}\ar[d]^{\ul{h}}&B'\ar[r]^-{\ul{v}}\ar[d]^{\ul{s}}&\Sigma^\X(A)\ar[d]^{\Sigma^\X(\ul{f})}\\
&A'\ar@{=}[r]\ar[d]^{\ul{n}}&A'\ar[r]^-{\ul{n}}\ar[d]^{\ul{t}}&\Sigma^\X(B)\\
&\Sigma^\X(B)\ar[r]^-{\Sigma^\X(\ul{l})}&\Sigma^\X(C')}
\endxy\]
	such that the second column from the right is in $\Delta^\X$. Thus we have a commutative diagram in $\G/\X$:
	\[\xy\xymatrixcolsep{2pc}\xymatrix@C42pt@R16pt{
A\ar[r]^-{\ul{f}}\ar@{=}[d]&B\ar[r]^-{R(\ul{l})}\ar[d]^{\ul{g}}&R(C')\ar[r]^-{R(\ul{m})}\ar[d]^{R(\ul{r})}& G^\X(A)\ar@{=}[d]\\
A\ar[r]^-{\ul{g\circ f}}&C\ar[r]^-{R(\ul{u})}\ar[d]^{R(\ul{h})}&R(B')\ar[r]^-{R(\ul{v})}\ar[d]^{R(\ul{s})}&G^\X(A)\ar[d]^{G^\X(\ul{f})}\\
&R(A')\ar@{=}[r]\ar[d]^{R(\ul{n})}&R(A')\ar[r]^-{R(\ul{n})}\ar[d]^{R(\Sigma^\X(\ul{j}^{C'})\circ \ul{t})}&G^\X(B)\\
&G^\X(B)\ar[r]^-{R(\Sigma^\X(\ul{j}^{C'})\circ \ul{l})}&G^\X(R(C'))}
\endxy\]
By Lemma \ref{distriangle}, the second column from the right is in $\Delta^\X(\G)$.
\end{proof}
Dually, let $(\A, \Omega, \rmL(\F), \X)$ be a partial left triangulated category and $\G$ an additive subcategory of $\A$ satisfying $\X\subseteq \G\subseteq \F$. Then $\G$ is said to be {\it special precovering} in $\F$ if for each $A\in \F$, there is a left $\F$-triangle $\Omega(A)\to W_A\xrto{\omega_A} Q(A)\xrto{r_A} A$
such that $r_A$ is a $\G$-precover with $\omega_A$ as a weak kernel and $W_A\in \G^{\perp_{\F/\X}}$ (we fix one). If $\G$ is a special precovering subcategory of $\F$, then by Remark \ref{functor}, the embedding $E_\F\colon \G/\X\hookrightarrow \F/\X$ has a right adjoint $Q\colon \F/\X\to \G/\X$ by sending $A$ to $Q(A)$ and $\ul{f}\colon A\to B$ to $\ul{\check{f}}\colon Q(A)\to Q(B)$, where $\check{f}$ is as defined in Lemma \ref{lem:creplacement}.
Furthermore, by Theorem \ref{thm:main}, the subfactor category $\F/\X$ has a left triangulated structure $(\Omega_\X, \nabla_\X)$. Similar to the Definition \ref{stabilizing}, $\G$ is said to be {\it stabilizing} if for any diagram of the form
\[\xy\xymatrixcolsep{2pc}\xymatrix@C16pt@R16pt{
\Omega_\X(Q(B))\ar[r] \ar[d]_{\Omega(\ul{r}_B)} & K\ar[r]\ar@{.>}[d]^{\ul{s}}&Q(A)\ar[r]^-{\ul{\check{f}}}\ar[d]^-{\ul{r}_A}& Q(B)\ar[d]^{\ul{r}_B}\\
\Omega_\X(B)\ar[r]&L\ar[r]&A\ar[r]^-{\ul{f}}&B}
\endxy\]
where the rows are the distinguished left triangles in $\nabla_\X$, there is a morphism $\ul{s}\colon K\to L$ such that the whole diagram commutative and $Q(\ul{s})$ is an isomorphism in $\G/\X$. \noindent Note that $\F$ is always stabilizing in $(\A, \Omega, \rmL(\F),\X)$.

 Denote by $H_\X=Q\circ \Omega_\X\circ E_\F$, and for each morphism $f\colon A\to B$ in $\G$, we call the left $H_\X$-sequence $H_\X(B)\xrto{Q(\ul{v})} Q(K)\xrto{Q(\ul{u})} A\xrto{Q(\ul{f})} B$ a {\it standard left triangle} in $\G/\X$ if there is a left triangle $\Omega_\X(B)\xrto{\ul{v}} K\xrto{\ul{u}} A\xrto{\ul{f}} B$ in $\nabla_\X$. We use $\nabla_\X(\G)$ to denote the class of the left $H_\X$-sequences which are isomorphic to standard left triangles in $\G/\X$. We have the dual of Proposition \ref{rstabilizing}:
\begin{proposition} \label{lstabilizing} Let $(\A, \Omega, \rmL(\F),\X)$ be a partial left triangulated category with $\G$ a stabilizing subcategory. Then $(\G/\X, H_\X, \nabla_\X(\G))$ is a left triangulated category.
\end{proposition}

\section{Pre-partial triangulated categories}
In this section we introduce the new concept of a pre-partial triangulated category and show that it is a proper setting to construct pre-triangulated categories.
\subsection*{Pre-triangulated categories} We recall the definition of a pre-triangulated category:
\begin{definition}  \label{defn:ptc} (\cite[Definition II.1.1]{Beligiannis/Reiten07}) Let $\T$ be an additive category admitting a right triangulated structure $(\Psi, \Delta)$ and a left triangulated structure $(\Phi, \nabla)$. The 5-tuple $(\T,\Psi, \Phi, \Delta, \nabla)$ is called a {\it pre-triangulated category} if the following conditions hold:

(a)  $(\Psi, \Phi)$ is an adjoint pair.

(b) For any diagram in $\T$ with commutative left square:
\[\xy\xymatrixcolsep{2pc}\xymatrix@C16pt@R16pt{
A\ar[r]\ar[d]_\alpha &B\ar[r]\ar[d]^{\beta}&C\ar[r]\ar@{.>}[d]^\gamma & \Psi(A)\ar[d]^{\Psi(\alpha)\circ \varepsilon_{C'}}\\
\Phi(C')\ar[r]&A'\ar[r]&B'\ar[r]& C'}
\endxy\]
where the upper row is in $\Delta$, the lower row is in $\nabla$, and $\varepsilon_{C'}$ is the counit of the adjoint pair $(\Psi, \Phi)$, there exists a morphism $\gamma\colon C\to B'$ making the diagram commutative.

(c) For any diagram in $\T$ with commutative right square:
\[\xy\xymatrixcolsep{2pc}\xymatrix@C16pt@R16pt{
A\ar[r]\ar[d]_{\eta_A\circ \Phi(\alpha)} &B\ar[r]\ar@{.>}[d]^{\gamma}&C\ar[r]\ar[d]^\beta & \Psi(A)\ar[d]^{\alpha}\\
\Phi(C')\ar[r]&A'\ar[r]&B'\ar[r]& C'}
\endxy\]
where the upper row is in $\Delta$, the lower row is in $\nabla$, and $\eta_A$ is the unit of the adjoint pair $(\Psi, \Phi)$, there exists a morphism $\gamma\colon B\to A'$ making the diagram commutative.
\end{definition}

We give an equivalent definition of a pre-triangulated category which is easier to be verified:

\begin{lemma} \ \label{lem:ptc} Let $\T$ be an additive category admitting a right triangulated structure $(\Psi, \Delta)$ and a left triangulated structure $(\Phi, \nabla)$ such that $(\Psi, \Phi)$ is an adjoint pair. Then $(\T,\Psi, \Phi, \Delta, \nabla)$ is a pre-triangulated category if and only if the following two conditions hold:

$(\mathrm{b}')$  For any diagram in $\T$:
\[\xy\xymatrixcolsep{2pc}\xymatrix@C16pt@R16pt{
\Phi(N)\ar[r]^-u\ar@{=}[d]&L\ar[r]^-{v'}\ar@{=}[d]&M'\ar[r]^-{w'}\ar@{.>}[d]^\gamma& \Psi(\Phi(N))\ar[d]^{\varepsilon_N}\\
\Phi(N)\ar[r]^-u&L\ar[r]^-v&M\ar[r]^-w& N}
\endxy\]
where the upper row is in $\Delta$, the lower row is in $\nabla$, and $\varepsilon_N$ is the counit of the adjoint pair, there exists a morphism $\gamma\colon M'\to M$ making the whole diagram commutative.

$(\mathrm{c}')$  For any diagram in $\T$:
\[\xy\xymatrixcolsep{2pc}\xymatrix@C16pt@R16pt{
A\ar[r]^-f\ar[d]_{\eta_A}&B\ar[r]^-g\ar@{.>}[d]^\gamma&C\ar[r]^-h\ar@{=}[d]& \Psi(A)\ar@{=}[d]\\
\Phi(\Psi(A))\ar[r]^-{f'}&B'\ar[r]^-{g'}&C\ar[r]^-h& \Psi(A)}
\endxy\]
where the upper row is in $\Delta$, the lower row is in $\nabla$, and $\eta_A$ is the unit of the adjoint pair, there exists a morphism $\gamma\colon B\to B'$ making the whole diagram commutative.
\end{lemma}

\begin{proof} If $(\T,\Psi, \Phi, \Delta, \nabla)$ is a pre-triangulated category, then both the condition $(\mathrm{b}')$ and $(\mathrm{c}')$ hold since they are special cases of $(\mathrm{b})$ and$(\mathrm{c})$ of Definition \ref{defn:ptc}. Conversely, assume that the conditions $(\mathrm{b}')$ and $(\mathrm{c}')$ hold. We will show that Definition \ref{defn:ptc} (b)-(c) are satisfied.

For Definition \ref{defn:ptc} (b), consider the following diagram in $\T$ with commutative left square:
\begin{equation} \label{condition b}
\xy\xymatrixcolsep{2pc}\xymatrix@C16pt@R16pt{
A\ar[r]^f\ar[d]_\alpha &B\ar[r]^g\ar[d]^{\beta}&C\ar[r]^-h & \Psi(A)\ar[d]^{\Psi(\alpha)\circ \varepsilon_{C'}}\\
\Phi(C')\ar[r]^-{f'}&A'\ar[r]^{g'}&B'\ar[r]^{h'}& C'}
\endxy
\end{equation}
where the upper row is in $\Delta$, the lower row is in $\nabla$, and $\varepsilon_{C'}$ is the counit of the adjoint pair $(\Psi, \Phi)$. Embedding $f'$ into a right triangle $\Phi(C')\xrto{f'} A'\xrto{l} D\xrto{m}\Psi(\Phi(C'))$ in $\Delta$. Then there is a commutative diagram of right triangles in $\Delta$ by (RT3):
\[\xy\xymatrixcolsep{2pc}\xymatrix@C16pt@R16pt{
A\ar[r]^f\ar[d]_\alpha &B\ar[r]^g\ar[d]^{\beta}&C\ar[r]^-h \ar[d]^n & \Psi(A)\ar[d]^{\Psi(\alpha)}\\
\Phi(C')\ar[r]^-{f'}&A'\ar[r]^{l}&D\ar[r]^-{m}& \Psi(\Phi(C'))}
\endxy\]
By the condition $(\mathrm{b}')$, we have the following commutative diagram:
\[\xy\xymatrixcolsep{2pc}\xymatrix@C16pt@R16pt{
\Phi(C')\ar[r]^-{f'}\ar@{=}[d]&A'\ar[r]^-{l}\ar@{=}[d]&D\ar[r]^-{m}\ar[d]^w& \Psi(\Phi(C'))\ar[d]^{\varepsilon_{C'}}\\
\Phi(C')\ar[r]^-{f'}&A'\ar[r]^-{g'}&B'\ar[r]^-{h'}& C'}
\endxy\]
Thus the composite $w\circ n$ is the desired filler in (\ref{condition b}). Similarly, we can show that Definition \ref{defn:ptc} (c) holds.
\end{proof}

\begin{remark} By the definition of a pre-triangulated category, we know that if the right triangulated structure $(\Psi, \Delta)$ or the left triangulated structure $(\Phi, \nabla)$ is a triangulated structure in the pre-triangulated category $(\T, \Psi, \Phi, \Delta, \nabla)$, then $\nabla=\Delta$.
\end{remark}

\subsection*{Pre-partial triangulated categories}
\begin{definition} \label{defn:pptc}\ Let $(\A, \Sigma, \R(\C), \X)$ be a partial right triangulated category and $(\A, \Omega, \rmL(\F), \X)$ a partial left triangulated category. Then the 6-tuple $(\A, \Omega, \Sigma, \rmL(\F), \R(\C), \X)$ is called a {\it  pre-partial triangulated category} if the following conditions hold:

(a)  $(\Omega, \Sigma)$ is an adjoint pair (we use $\psi$ to denote the adjunction isomorphism).

(b)  $\C\cap \F$ is stabilizing in both $(\A, \Sigma, \R(\C), \X)$ and $(\A, \Omega, \rmL(\F), \X)$.

(c) Given a diagram
 \[\xy\xymatrixcolsep{2pc}\xymatrix@C16pt@R16pt{
& A\ar[r]^-{f}&B\ar[r]^-{g}\ar[d]^{\beta}& C\ar[r]^{h}&\Sigma(A)\\
\Omega(N)\ar[r]^-{u}&L\ar[r]^-{v}& M\ar[r]^{w}&N&}
\endxy\]
where the upper row is a right $\C$-triangle with $f$ an $\X$-monic and the lower row is a left $\F$-triangle with $w$ an $\X$-epic. Then

(i) if there is a morphism $\gamma\colon C\to N$ such that $\gamma\circ g=w\circ \beta$, then there exists $\alpha\colon A\to L$ such that $v\circ \alpha=\beta\circ f$ and $u\circ \Omega(\gamma)=-\alpha\circ \psi^{-1}_{C,A}(h)$;

(ii) if there is a morphism $\alpha\colon A\to L$ such that $v\circ \alpha=\beta\circ f$, then there exists $\gamma\colon C\to N$ such that $\gamma\circ g=w\circ\beta$ and $\Sigma(\alpha)\circ h=-\psi_{N,L}(u)\circ \gamma$.

(d) For any commutative diagram
\[\xy\xymatrixcolsep{2pc}\xymatrix@C16pt@R16pt{
& A\ar[r]^-{i^A}\ar[d]_{\alpha}&X^A\ar[r]^-{p^A}\ar[d]^{\beta}& U^A\ar[d]^\gamma\ar[r]^{q^A}&\Sigma(A)\\
\Omega(N)\ar[r]^-{u}&L\ar[r]^-{v}& M\ar[r]^{w}&N&}
\endxy\]
such that $u\circ \Omega(\gamma)=-\alpha\circ \psi^{-1}_{U^A,A}(q^A)$, where the upper row is the fixed right $\C$-triangle for $A\in \C\cap \F$ and the lower row is in $\rmL(\F)$ with $w$ an $\X$-epic, then $\alpha$ factors through $i^A$ if $\gamma$ factors through $w$.

(e)  For any commutative diagram
\[\xy\xymatrixcolsep{2pc}\xymatrix@C16pt@R16pt{
& A\ar[r]^-{f}\ar[d]_{\alpha}&B\ar[r]^-{g}\ar[d]^{\beta}& C\ar[d]^\gamma\ar[r]^-{h}&\Sigma(A)\\
\Omega(D)\ar[r]^-{\nu_D}&U_D\ar[r]^-{\iota_D}& X_D\ar[r]^{\pi_D}&D&}
\endxy\]
such that $\Sigma(\alpha)\circ h=-\psi_{D,U_D}(\nu_D)\circ \gamma$, where the lower row is the fixed left $\F$-triangle for $D\in \C\cap \F$ and the upper row is in $\R(\C)$ with $f$ an $\X$-monic, then $\gamma$ factors through $\pi_D$ if $\alpha$ factors through $f$.

\end{definition}

Now assume that $(\A, \Omega, \Sigma, \rmL(\F), \R(\C), \X)$ is a pre-partial triangulated category. Then the subfactor category $(\C\cap\F)/\X$ has a right triangulated structure $(G^\X, \Delta^\X(\C\cap \F))$ as constructed in Proposition \ref{rstabilizing} and a left triangulated structure $(H_\X, \nabla_\X(\C\cap \F))$ as constructed in Proposition \ref{lstabilizing}. We have the following result:

\begin{lemma} \label{lem:adjointpair} If $(\A, \Omega, \Sigma, \rmL(\F), \R(\C), \X)$ is a pre-partial triangulated category, then $(G^\X, H_\X)$ is an adjoint pair on $(\C\cap\F)/\X$.
\end{lemma}
\begin{proof} For $A, B\in (\C\cap \F)/\X$, recall that $\Sigma^\X(A)=U^A$ where $U^A$ is defined by the fixed right $\C$-triangle $A\xrto{i^A} X^A\xrto{p^A}U^A\xrto{q^A}\Sigma(A)$ for $A$, and $\Omega_\X(B)=U_B$ is determined by the fixed left $\F$-triangle $\Omega(B)\xrto{\nu_B} U_B\xrto{\iota_B}X_B\xrto{\pi_B} B$ for $B$. Since $(R,E^\C)\colon \C/\X\to (\C\cap\F)/\X$ and $(E_\F, Q)\colon (\C\cap \F)/\X\to \F/\X$ are adjoint pairs by Remark \ref{functor}, we have the following natural isomorphisms:
$$\Hom_{(\C\cap \F)/\X}(G^\X(A), B)\xrto{(\ul{j}^{U^A})^*} \Hom_{\C/\X}(U^A,B)$$
induced by $j^{U^A}\colon U_A\to R(U^A)$ and
$$\Hom_{(\C\cap \F)/\X}(A, H_\X(B))\xrto{(\ul{r}_{U_B})_*} \Hom_{\F/\X}(A,U_B)$$
induced by $r_{U_B}\colon Q(U_B)\to U_B$.

We claim that there is a bijection $\varphi_{A,B}\colon \Hom_{\C/\X}(U^A,B)\to \Hom_{\F/\X}(A, U_B)$ which is natural in $A,B$ by sending $\ul{f}\in \Hom_{\C/\X}(\Sigma^\X(A),B)$ to $-\ul{\kappa}$, where $\kappa$ is constructed by applying Definition \ref{defn:pptc} (c)(i) to the following diagram
\[\xy\xymatrixcolsep{2pc}\xymatrix@C16pt@R16pt{
&A\ar[r]^-{i^A}\ar[d]_{\kappa}&X^A\ar[r]^-{p^A}\ar[d]^{\sigma}& U^A\ar[d]^f\ar[r]^{q^A}&\Sigma(A)\\
\Omega(B)\ar[r]^-{\nu_B}&U_B\ar[r]^-{\iota_B}& X_B\ar[r]^{\pi_B}&B&}
\endxy\]
where the existence of $\sigma$ is since $\pi_B$ is an $\X$-precover. That is, $\kappa$ satisfies $\sigma\circ i^A=\iota_B\circ \kappa$ and $-\kappa\circ \psi_{U^A,A}^{-1}(q^A)=\nu_B\circ \Omega(f)$. By Definition \ref{defn:pptc} (d), if $f$ factors through $\pi_B$, then $\kappa$ factors through $i^A$, and thus $\ul{\kappa}$ is determined uniquely by $\ul{f}$. 
By the naturality of $\psi_{B, U_B}$ in $B$ and $U_B$, we have $-\psi_{B,U_B}(\nu_B)\circ f=-\psi_{B,U_B}(\nu_B\circ \Omega(f))=\psi_{B,U_B}(\kappa\circ \psi_{U^A,A}^{-1}(q^A))=\Sigma(\kappa)\circ q^A$ (see, for example, \cite[Page 81, (3)-(4)]{MacLane1}). Thus if $\kappa$ factors through $i^A$, then $f$ factors through $\pi_B$ by Definition \ref{defn:pptc} (e), and so $\ul{f}$ is also determined uniquely by $\ul{\kappa}$. So $\varphi_{A,B}$ is a bijection. The naturality of $\varphi_{A,B}$ in $A,B$ can be verified directly. Therefore we have a bijection
$$\phi_{A,B}=(\ul{r}_{U_B})_*^{-1}\circ \varphi_{A,B}\circ (\ul{j}^{U^A})^*\colon \Hom_{(\C\cap \F)/\X}(G^\X(A), B)\to \Hom_{(\C\cap \F)/\X}(A, H_\X(B))$$
which is natural in $A,B\in (\C\cap \F)/\X$, and then $(G^\X,H_\X)$ is an adjoint pair on $ (\C\cap \F)/\X$.
\end{proof}

\begin{theorem}\label{thm:pretricat} If $(\A, \Omega, \Sigma, \rmL(\F), \R(\C), \X)$ is a pre-partial triangulated category, then \\ $( (\C\cap \F)/\X,
G^\X, H_\X, \Delta^\X(\C\cap \F), \nabla_\X(\C\cap \F))$ is a pre-triangulated category.
\end{theorem}
\begin{proof} \ For simplicity, let $\G=\C\cap \F$. By Lemma \ref{lem:adjointpair}, we know that $(G^\X, H_\X)$ is an adjoint pair.

We verify the conditions $(\mathrm{b}')$ and $(\mathrm{c}')$ in Lemma \ref{lem:ptc}. For $(\mathrm{c}')$, without loss of generality, we may only consider standard right triangles in $\Delta^\X(\G)$ and standard left triangles $\nabla_\X(\G)$. Assume that we have a commutative diagram in $\G/\X$:
\begin{equation} \label{ptc(c)}
\xy\xymatrixcolsep{2pc}\xymatrix@C16pt@R16pt{
A\ar[r]^-{\ul{f}}\ar[d]_{\eta_A}&B\ar[r]^-{R(\ul{g})}&R(C)\ar[r]^-{R(\ul{h})}\ar@{=}[d]& G^\X(A)\ar@{=}[d]\\
H_\X(G^\X(A))\ar[r]^-{Q(\ul{u})}& Q(K)\ar[r]^-{Q(\ul{v})}&R(C)\ar[r]^-{R(\ul{h})}& G^\X(A)}
\endxy
\end{equation}
where the upper row is a standard right triangle in $\Delta^\X(\G)$, the lower row is a standard left triangle in $\nabla_\X(\G)$, and $\eta_A$ is the unit of the adjoint pair $(G^\X, H_\X)$. So there is a right triangle $A\xrto{\ul{f}}B\xrto{\ul{g}}C\xrto{\ul{h}}\Sigma^\X(A)$ in $\Delta^\X$ and we may assume that it is induced by the following commutative diagram of right $\C$-triangles:
\[\xy\xymatrixcolsep{2pc}\xymatrix@C16pt@R16pt{
A\ar[r]^-{\left(\begin{smallmatrix}
i^A \\
f
\end{smallmatrix}\right)}\ar@{=}[d]&X^A\oplus B\ar[r]^-{(\theta, g)} \ar[d]^{(1,0)}& C\ar[r]^-r\ar[d]^h& \Sigma(A)\ar@{=}[d]\\
A\ar[r]^-{i^A}& X^A\ar[r]^-{p^A}&\Sigma^\X(A)\ar[r]^-{q^A}& \Sigma(A)}
\endxy\]
Similarly, there is a left triangle $\Omega_\X(G^\X(A))\xrto{\ul{u}}K\xrto{\ul{v}}R(C)\xrto{R(\ul{h})} G^\X(A)$ in $\nabla_\X$ which is induced by the following commutative diagram of left $\F$-triangles:
\[\xy\xymatrixcolsep{2pc}\xymatrix@C28pt@R16pt{
\Omega(G^\X(A))\ar[r]^-{\nu_{G^\X(A)}}\ar@{=}[d]&\Omega_\X(G^\X(A))\ar[r]^-{\iota_{G^\X(A)}} \ar[d]^u&X_{G^\X(A)}\ar[r]^-{\pi_{G^\X(A)}}\ar[d]^{\left(\begin{smallmatrix}
1 \\
0
\end{smallmatrix}\right)}& G^\X(A)\ar@{=}[d]\\
\Omega(G^\X(A))\ar[r]^-w&K\ar[r]^-{\left(\begin{smallmatrix}
\vartheta \\
v
\end{smallmatrix}\right)}&X_{G^\X(A)}\oplus R(C)\ar[r]^-{(\pi_{G^\X(A)}, \hat{h})}& G^\X(A)}
\endxy\]
By \cite[Diagram (3.4)]{ZWLi2}, $B\xrto{-g}C\xrto{h}\Sigma^\X(A)\xrto{\Sigma(f)\circ q^A}\Sigma(B)$ is a right $\C$-triangle with $g$ an $\X$-monic. 
Then by Definition \ref{defn:pptc} (c)(i), there is a morphism $\alpha\colon B\to K$ such that the following diagram is commutative
\[\xy\xymatrixcolsep{2pc}\xymatrix@C28pt@R18pt{
&B\ar[r]^-{-g} \ar[d]_{\alpha}&C\ar[r]^-{h}\ar[d]^{\left(\begin{smallmatrix}
0 \\
j^C
\end{smallmatrix}\right)}& \Sigma^\X(A)\ar[d]^{j^{\Sigma^\X(A)}}\ar[r]^{\Sigma(f)\circ q^A}& \Sigma(B)\\
\Omega(G^\X(A))\ar[r]^-{w}&K\ar[r]^-{\left(\begin{smallmatrix}
\vartheta \\
v
\end{smallmatrix}\right)}&X_{G^\X(A)}\oplus R(C)\ar[r]^-{(\pi_{G^\X(A)}, \hat{h})}& G^\X(A)&}
\endxy\]
and $w\circ \Omega(j^{\Sigma^\X(A)})=-\alpha\circ \psi^{-1}_{\Sigma^\X(A),B}(\Sigma(f)\circ q^A)$. Since $j^C\circ g=\hat{g}$ and $Q|_{\G/\X}=\id_{\G/\X}$, we know that $Q(\ul{v})\circ Q(\ul{\alpha})=-Q(R(\ul{g}))=-R(\ul{g})$.

By the proof of Lemma \ref{lem:adjointpair}, there is a commutative diagram
\[\xy\xymatrixcolsep{2pc}\xymatrix@C18pt@R18pt{
&A\ar[r]^-{i^A} \ar[d]_{\mu}&X^A\ar[r]^-{p^A}\ar[d]^{\delta}& \Sigma^\X(A)\ar[d]^{j^{\Sigma^\X(A)}}\ar[r]^{q^A}& \Sigma(A)\\
\Omega(G^\X(A))\ar[r]^-{\nu_{G^\X(A)}}&\Omega_\X(G^\X(A))\ar[r]^-{\iota_{G^\X(A)}}&X_{G^\X(A)}\ar[r]^-{\pi_{G^\X(A)}}& G^\X(A)&}
\endxy\]
such that $-\mu \circ \psi^{-1}_{\Sigma^\X(A),A}(q^A)=\nu_{G^\X(A)}\circ \Omega(j^{\Sigma^\X(A)})$. Then we have the following commutative diagram
\[\xy\xymatrixcolsep{2pc}\xymatrix@C40pt@R20pt{
&A\ar[r]^-{i^A} \ar[d]_{u\circ \mu-\alpha\circ f}&X^A\ar[r]^-{p^A}\ar[d]^{\left(\begin{smallmatrix}
\delta \\
-j^C\circ \theta
\end{smallmatrix}\right)}& \Sigma^\X(A)\ar[d]^{0}\ar[r]^{q^A}& \Sigma(A)\\
\Omega(G^\X(A))\ar[r]^-w&K\ar[r]^-{\left(\begin{smallmatrix}
\vartheta \\
v
\end{smallmatrix}\right)}&X_{G^\X(A)}\oplus R(C)\ar[r]^-{(\pi_{G^\X(A)}, \hat{h})}& G^\X(A)&}
\endxy\]
 and $u\circ \mu-\alpha\circ f$ satisfies \begin{align*}(u\circ \mu-\alpha\circ f)\circ \psi^{-1}_{\Sigma^\X(A),A}(q^A)&=u\circ \mu\circ \psi^{-1}_{\Sigma^\X(A),A}(q^A)-\alpha\circ f\circ\psi^{-1}_{\Sigma^\X(A),A}(q^A)\\
&=-u\circ \nu_{\G^\X(A)}\circ \Omega(j^{\Sigma^\X(A)})-\alpha\circ \psi^{-1}_{\Sigma^\X(A),A}(\Sigma(f)\circ q^A)\\
&=-w\circ \Omega(j^{\Sigma^\X(A)})+w\circ \Omega(j^{\Sigma^\X(A)})\\
&=0.
\end{align*}
So $u\circ \mu-\alpha\circ f$ factors through $i^A$ by Definition \ref{defn:pptc} (d), then $\ul{u}\circ\ul{\mu}=\ul{\alpha}\circ \ul{f}$. Since the unit $\eta_A$ of the adjoint pair $(G^\X, H_\X)$ is $(\ul{r}_{\Omega_\X(G^\X(A))})_*^{-1}\circ \varphi_{A,G^\X(A)}\circ (\ul{j}^{\Sigma^\X(A)})^*(1_{G^\X(A)})=-Q(\ul{\mu})$. Therefore $-Q(\ul{\alpha})\colon B\to Q(K)$ is a desired filler in (\ref{ptc(c)}), and we are done.

Similarly, we can prove Lemma \ref{lem:ptc} $(\mathrm{b}')$. \end{proof}

\begin{lemma} \label{lem:strongpptc} Let $(\A, \Sigma, \R(\C),\X)$ be a partial right triangulated category and $(\A, \Omega, \rmL(\F), \X)$ a partial left triangulated category. Assume that Definition \ref{defn:pptc} $(\mathrm a)$-$(\mathrm b)$ and the following two conditions hold:

$(\mathrm g)$  if $\Omega(B)\xrto{u} K\xrto{v} A\xrto{f}B$ is in $\rmL(\F)$ with $A,B \in \C\cap \F$, then $ K\xrto{v} A\xrto{f} B\xrto{-\psi_{B,K}(u)} \Sigma(K)$ is in $\R(\C)$;

$(\mathrm h)$  if $A\xrto{f} B\xrto{g} C\xrto{h}\Sigma(A)$ is in $\R(\C)$ with $A,B \in \C\cap \F$, then $ \Omega(C)\xrto{-\psi_{C,A}^{-1}(h)} A\xrto{f} B\xrto{g} C$ is in $\rmL(\F)$;

\noindent then $(\A, \Omega, \Sigma, \rmL(\F), \R(\C), \X)$ is a pre-partial triangulated category.
\end{lemma}

\begin{proof} In this case, Definition \ref{defn:pptc} (c) can be proved by Definition \ref{def:prtc} (PRT3) and its dual, and Definition \ref{defn:pptc} (d)-(e) can be proved by Definition \ref{def:prtc} (PRT2) and its dual. We leave the details to the reader.
\end{proof}

By Lemma \ref{lem:strongpptc}, a {\it partial triangulated category} in the sense of \cite[Definition 6.3]{ZWLi2} is a pre-partial triangulated category.

\begin{example} (i)  Let $(\T, [1],\Delta)$ be a triangulated category and $\X$ an additive subcategory of $\T$. Assume that every object $A$ in $\T$ has both an $\X$-preenvelope $A\to X^A$ and an $\X$-precover $X_A\to A$. Then there are a partial right triangulated category $(\T,[1],\R(\T),\X)$ by \cite[Proposition 4.7]{ZWLi2} and a partial left triangulated category $(\T,[-1],\rmL(\T), \X)$ by \cite[Proposition 4.9]{ZWLi2}. Since both $\R(\T)$ and $\rmL(\T)$ consist of triangles in $\Delta$ and $\T$ is a triangulated category, we know that Lemma \ref{lem:strongpptc} (g)-(h) hold. So $(\T, [-1], [1], \rmL(\T), \R(\T), \X)$ is a two-sided pseudo-triangulated category since $\T$ is always stabilizing in both $(\T,[1],\R(\T), \X)$ and $(\T,[-1],\rmL(\T), \X)$. In particular, $\T/\X$ is a pre-triangulated category by Theorem \ref{thm:pretricat}, this is just the \cite[Theorem 2.2]{Jorgensen}.

(ii)  Let $\C$ be an additive category and $\X$ an additive subcategory of $\C$. Assume that every object in $\C$ has both an $\X$-preenvelope and an $\X$-precover, any $\X$-monic has a cokernel and any $\X$-epic has a kernel in $\C$. Then there are a partial right triangulated category $(\C,0,\R(\C), \X)$ by \cite[Proposition 4.1]{ZWLi2} and a partial left triangulated category $(\C,0,\rmL(\C), \X)$ by \cite[Proposition4.2]{ZWLi2}. In this case, a right $\C$-triangle is just a right exact sequence and a left $\C$-triangle is just a left exact sequence, so Definition \ref{defn:pptc} (c)-(e) hold. Moreover, $\C$ is stabilizing both in $(\C,0,\R(\C), \X)$ and $(\C,0,\rmL(\C), \X)$. Then $(\C,0,0, \rmL(\C), \R(\C), \X)$ is a pre-partial triangulated category. In particular, the stable category $\C/\X$ is a pre-triangulated category by Theorem \ref{thm:pretricat}, this is just the \cite[Corollary 4.10]{Beligiannis01}.

(iii) Let $A$ be an Artin algebra. Denote by $\C=\mathrm{mod}A$ the category of finitely generated right $A$-modules and $\X=\mathcal{P}(A)$ the subcategory of finitely generated projective $A$-modules. Then every object in $\C$ has an $\X$-precover and an $\X$-preenvelope, and thus $(\C, 0,0, \rmL(\C),  \R(\C), X)$ is a pre-partial triangulated category, and the corresponding stable category $\C/\X=\ul{\mathrm{mod}}A$ is a pre-triangulated category.
\end{example}

\section{The homotopy theory of Hovey triples}
In this section we first recall some basic facts and notions for exact categories and then show that the homotopy theory of exact model structures can be covered by our homotopy theory.
\subsection*{Exact categories}
Let $\A$ be an additive category. A {\it kernel-cokernel sequence} in $\A$ is a sequence $A\stackrel{i}\to B\stackrel{d}\to C $
such that $i=\Ker d$ and $d=\Coker i$. Let $\mathcal{E}$ be a class of kernel-cokernel sequences of $\A$. Following Keller \cite[Appendix A]{Keller90}, we call a kernel-cokernel sequence a {\it conflation} if it is in $\mathcal{E}$. A morphism $i$ is called an {\it inflation} if there is a conflation $ A\stackrel{i}\to B\stackrel{d} \to C$ and the morphism $d$ is called a {\it deflation}.

Recall that an {\it exact structure} on an additive category $\A$ is a class $\E$ of kernel-cokernel sequences which is closed under isomorphisms and satisfies the following axioms due to Quillen \cite{Quillen73} and Keller \cite[Appendix A]{Keller90}:
\vskip5pt
$(\mathrm {Ex0})$ \ The identity morphism of the zero object is an inflation.

$(\mathrm{Ex1})$  The class of deflations is closed under composition.

$(\mathrm{Ex1})^{\mathrm{op}}$  The class of inflations is closed under composition.

$(\mathrm{Ex2})$  For any deflation $d\colon A\to B$ and any morphism $f\colon B'\to B$, there exists a pullback diagram such that $d'$ is a deflation:
\[\xy\xymatrixcolsep{2pc}\xymatrix@C16pt@R16pt{
A'\ar@{.>}[r]^{d'}\ar@{.>}[d]_{f'}&B'\ar[d]^f\\
A\ar[r]^d&B}
\endxy\]

$(\mathrm{Ex2})^{\mathrm{op}}$  For any inflation $i\colon C\to D$ and any morphism $g\colon C\to C'$, there is a pushout diagram such that $i'$ is an inflation:	
\[\xy\xymatrixcolsep{2pc}\xymatrix@C16pt@R16pt{
C\ar[r]^i\ar[d]_g&D\ar@{.>}[d]^{g'}\\
C'\ar@{.>}[r]^{i'}&D'}
\endxy\]	

An {\it exact category} is a pair $(\A, \E)$ consisting of an additive category $\A$ and an exact structure $\E$ on $\A$.  We refer the reader to \cite{Buhler10} for a readable introduction to exact categories. Sometimes we suppress the class of $\mathcal{E}$ and just say that $\A$ is an exact category.

In an exact category $(\A, \mathcal{E})$, we can define the Yoneda Ext bifunctor $\Ext^1_\A(C,A)$. It is the abelian group of equivalence classes of conflations $ A\to  B\to C $ in $\mathcal{E}$; see \cite[Chapter XII.4]{MacLane} for details.
\subsection*{Hovey triples in exact categories}
\begin{definition}  (\cite[Definition 2.1]{Gillespie11}) Let $\A$ be an exact category. A {\it cotorsion pair} in $\A$ is a pair $(\C, \F)$ of classes of objects of $\A$ such that $\C=\{C\in \A \ | \ \Ext^1_\A(C, \F)=0\}$ and  $\F=\{F\in \A \ | \ \Ext^1_\A(\C, F)=0\}.$
\end{definition}
\noindent The cotorsion pair $(\C, \F)$ is called {\it complete} if it has {\it enough projectives}, i.e. for each $A\in \A$ there is a conflation $F\to  C\to A$ such that $C\in \C, F\in \F$, and {\it enough injectives}, i.e. for each $A\in \A$, there is a conflation $A\to F'\to C'$ such that $C'\in \C, F'\in \F$.

Follow \cite[Definition 3.1]{Gillespie12}, a triple $(\C, \mathcal{W}, \F)$ of classes of objects in an exact category $(\A, \E)$ is called a {\it Hovey triple} if $\mathcal{W}$ is a thick subcategory of $\A$ and both $(\C, \mathcal{W}\cap \F)$ and $(\C\cap \mathcal{W}, \F)$ are complete cotorsion pairs in $\A$.  Recall that the  subcategory $\mathcal{W}$ is called {\it thick} if it is closed under direct summands and if two out of three of the terms in a conflation are in $\mathcal{W}$, then so is the third.

If $\mathcal{L}$ is an additive subcategory of the exact category $(\A,\E)$, we will always denote by
$$\R(\mathcal{L})=\{ A\xrto{f} B\to C \to 0 \ | \  A\xrto{f} B\to C \in \E,  C\in \mathcal{L} \}$$
and
 $$\rmL(\mathcal{L})=\{ 0\to K\to A\xrto{f}B \ |  \ K\to A\xrto{f} B \in \E, K\in \mathcal{L} \}.$$

Now let $(\C, \mathcal{W}, \F)$ be a Hovey triple in the exact category $(\A, \E)$. Let $\X=\C\cap \mathcal{W}\cap \F$. Then we have two partial right triangulated categories $(\A, 0, \R(\A), \F)$ and $(\A,0,\R(\C), \X)$, and two partial left triangulated categories $(\A,0,\rmL(\A), \C)$ and $(\A,0,\rmL(\F), \X)$ by \cite[Corollary 4.5]{ZWLi2}. In this case, we have
\begin{lemma} \ \label{lem:exacthomotopytriple} \ $(\C, \X, \F )$ is a localization triple in $\A$.
\end{lemma}
\begin{proof} Since $(\C,\mathcal{W}\cap\F)$ is a complete cotorsion pair, for each object $A\in \A$, there is a conflation $W_A\xrto{\omega_A} Q(A)\xrto{r_A} A$ such that $Q(A)\in \C$ and $W_A\in \mathcal{W}\cap \F$. In particular, $r_A$ is a $\C$-precover. We claim that $\mathcal{W}\cap \F\subseteq \C^{\perp_{\A/\X}}$. In fact, if $A\in \mathcal{W}\cap \F$, then $Q(A)\in \X=\C\cap \mathcal{W}\cap \F$ since $\mathcal{W}\cap \F$ is closed under extensions. For any morphism $f\colon C\to A$ with $C\in \C$, then $f$ factors through $r_A$ and then it factors through $Q(A)\in \X$. So $\mathcal{W}\cap \F\subseteq \C^{\perp_{\A/\X}}$, in particular, $W_A\in \C^{\perp_{\A/\X}}$. Similarly, since $(\C\cap \mathcal{W}, \F)$ is a complete cotorsion pair, for each object $A\in \A$ there is a conflation $A\xrto{j^A} R(A)\xrto{\tau^A} W^A$ with $R(A)\in \F$ and $W^A\in \C\cap \mathcal{W}\subseteq {^{\perp_{\A/\X}}\F}$. Since both $\C$ and $\F$ are closed under extensions, we know that if $A\in \F$, then $Q(A)\in \C\cap\F$, and if $A\in \C$, $R(A)\in \C\cap \F$. Thus $(\C,\X,\F)$ is a localization triple.
\end{proof}

The above Lemma shows that we can define the {\it homotopy category of the Hovey triple} $(\C, \mathcal{W}, \F)$ to be the corresponding homtopy category of the localization triple $(\C, \X, \F)$.  Moreover, the embedding $E^\C\colon  (\C\cap \F)/\X\hookrightarrow \C/\X$ has a left adjoint $R|_{\C/\X}$ which is the restriction of $R\colon \A/\X\to \F/\X$ as constructed in Remark \ref{functor}. Dually, the embedding $E_\F\colon  (\C\cap \F)/\X\hookrightarrow \F/\X$ has a right adjoint $Q|_{\F/\X}$ which is the restriction of $Q\colon \A/\X\to \C/\X$ as constructed in Remark \ref{functor}.

\begin{lemma} \label{lem:exactstabilizing}\ $\C\cap\F$ is stabilizing in both $(\A,0,\R(\C), \X)$ and $(\A,0,\rm L(\F), \X)$.
\end{lemma}
\begin{proof} \ We only prove that $\C\cap\F$ is stabilizing in $(\A,0,\R(\C), \X)$ since the other case can be proved dually. By the proof of Lemma \ref{lem:exacthomotopytriple}, we know that $\C\cap \F$ is a special pre-enveloping subcategory of $\C$. Assume that we have a diagram of the form:
 \[\xy\xymatrixcolsep{2pc}\xymatrix@C14pt@R14pt{
A\ar[r]^-{\ul{f}}\ar[d]_{\ul{j}^A}&B\ar[r]\ar[d]^{\ul{j}^B}&C\ar[r]& \Sigma^\X(A)\ar[d]^{\Sigma^\X(\ul{j}^A)}\\
R(A)\ar[r]^-{\ul{\hat{f}}}&R(B)\ar[r]&D\ar[r]&\Sigma^\X(R(A))}
\endxy\]
where the rows are distinguished right triangles in $\C/\X$. By \cite[Proposition 1.1.11]{BBD} (Verdier's Exercise, also called $4\times 4$-Lemma in the literature), we have a diagram
\[\xy\xymatrixcolsep{2pc}\xymatrix@C14pt@R14pt{
A\ar[r]^-{\ul{f}}\ar[d]_{\ul{j}^A}&B\ar[r]^-{\ul{g}}\ar[d]^{\ul{j}^B}&C\ar[r]^-{\ul{h}}\ar[d]^{\ul{t}}& \Sigma^\X(A)\ar[d]^{\Sigma^\X(\ul{j}^A)}\\
R(A)\ar[r]^-{\ul{\hat{f}}}\ar[d]&R(B)\ar[r]^-{\ul{u}}\ar[d]&D\ar[r]^-{\ul{v}}\ar[d]&\Sigma^\X(R(A))\ar[d]\\
W^A\ar[r]^-{\ul{r}}\ar[d] & W^B\ar[r]^-{\ul{s}}\ar[d]&W\ar[r]\ar[d]&\Sigma^\X(W^A)\ar[d]\\
\Sigma^\X(A)\ar[r]&\Sigma^\X(B)\ar[r]&\Sigma^\X(C) \ar[r]& \Sigma^\X(\Sigma^\X(A))}
\endxy\]
such that the first three rows and the first three columns are right triangles in $\Delta^\X$. Moreover, every square is commutative except for the bottom right one. Without loss of generality, we may assume that the right triangle $W^A\xrto{\ul{r}}W^B\xrto{\ul{s}}W\to \Sigma^\X(W^A)$ is induced by the right $\C$-triangle $W^A\xrto{r} W^B\xrto{s}W\to 0$ (recall that this means that $W^A\xrto{r} W^B\xrto{s}W$ is a conflation in $\mathcal{E}$ with $W\in \C$), and $C\xrto{\ul{t}}D\to W\to \Sigma^\X(C)$ is induced by $C\xrto{t} D\to W\to 0$. Since $W^A, W^B\in \mathcal{W}$ and $\mathcal{W}$ is thick, we know that $W\in \C\cap \mathcal{W}$. We claim that $R|_{\C/\X}(\ul{t})$ is an isomorphism in $ (\C\cap \F)/\X$. In fact, apply \cite[Proposition 1.1.11]{BBD} again, we have the following diagram
\[\xy\xymatrixcolsep{2pc}\xymatrix@C14pt@R14pt{
C\ar[r]^-{\ul{t}}\ar[d]_{\ul{j}^C}&D\ar[r]\ar[d]^{\ul{j}^D}&W\ar[r]\ar[d]^{\ul{\alpha}}& \Sigma^\X(C)\ar[d]^{\Sigma^\X(\ul{j}^C)}\\
R(C)\ar[r]^-{\ul{\hat{t}}}\ar[d]&R(D)\ar[r]\ar[d]&W'\ar[r]\ar[d]^{\ul{\beta}}&\Sigma^\X(R(C))\ar[d]\\
W^C\ar[r]\ar[d] & W^D\ar[r]\ar[d]&W''\ar[r]\ar[d]&\Sigma^\X(W^C)\ar[d]\\
\Sigma^\X(C)\ar[r]&\Sigma^\X(D)\ar[r]&\Sigma^\X(W) \ar[r]& \Sigma^\X(\Sigma^\X(C))}
\endxy\]
such that the first three rows and the first three columns are right triangles in $\Delta^\X$. Similar to the previous discussion, we know that $W''\in \C\cap\mathcal{W}$. Without loss of generality, we assume that the right triangle $W\xrto{\ul{\alpha}} W'\xrto{\ul{\beta}}W''\to \Sigma^\X(W)$ is induced by the right $\C$-triangle $W\xrto{\alpha}W'\xrto{\beta}W''\to 0$ in $\R(\C)$. Then $W'$ is in $\C\cap \mathcal{W}$ since $\mathcal{W}$ is thick. Assume that the right triangle $R(C)\xrto{\ul{\hat{t}}} R(D)\to W'\to \Sigma^\X(R(C))$ is induced by the right $\C$-triangle $R(C)\xrto{\left(\begin{smallmatrix}
i^{R(C)}\\
\hat{t}
\end{smallmatrix}\right)}X^{R(C)}\oplus R(D)\to W'\to 0$, which splits since $(\C\cap \mathcal{W}, \F)$ is a cotorsion pair. In particular, $W'$ is a direct summand of $X^{R(C)}\oplus R(D)$ which is in $\C\cap \F$, and thus $W'$ is in $\C\cap \F\cap \mathcal{W}=\X$. So $R|_{\C/\X}(\ul{t})=\ul{\hat{t}}$ is an isomorphism in $(\C\cap\F)/\X$. Thus $\C\cap \F$ is stabilizing in $\C$.
\end{proof}

Now we can prove our main result for a Hovey triple.

\begin{theorem} \label{thm:homoHoveytriple} Let $(\C, \mathcal{W}, \F)$ be a Hovey triple in an exact category $(\A, \E)$. Denote by $\X=\C\cap \mathcal{W}\cap \F$. Then $(\A, 0, 0,, \rmL(\F), \R(\C), \X)$ is a pre-partial triangulated category. \end{theorem}
\begin{proof}  In this case, $(0,0)$ is an adjoint pair on $\A$ and by Lemma \ref{lem:exactstabilizing}, $\C\cap\F$ is stabilizing both in $(\A, 0, \R(\C), \X)$ and $(\A, 0, \rmL(\F), \X)$. So Definition \ref{defn:pptc} (a)-(b) hold. The conditions (c) (i)-(ii) of Definition \ref{defn:pptc} follow directly by the universal property of kernels and cokernels. For Definition \ref{defn:pptc} (d), assume that we have a commutative diagram
\[\xy\xymatrixcolsep{2pc}\xymatrix@C16pt@R16pt{
& A\ar[r]^-{i^A}\ar[d]_{\alpha}&X^A\ar[r]^-{p^A}\ar[d]^{\beta}& U^A\ar[d]^\gamma\ar[r]&0\\
0\ar[r]&L\ar[r]^-{v}& M\ar[r]^{w}&N&}
\endxy\]
 where the first row is the fixed right $\C$-triangle for $A\in \C\cap \F$ and the second row is in $\rmL(\F)$ with $w$ an $\X$-epic. If $\gamma$ factors through $w$, i.e. there is a morphism $s\colon U^A\to M$ such that $\gamma=w\circ s$. Then $w\circ(\beta-s\circ p^A)=w\circ \beta-w\circ s\circ p^A=w\circ \beta-\gamma\circ p^A=0$. Thus there is morphism $t\colon X^A\to L$ such that $\beta-s\circ p^A=v\circ t$ since $v$ is a kernel of $w$. So $v\circ \alpha=\beta\circ i^A=v\circ t\circ i^A+s\circ p^A\circ i^A=v\circ t\circ i^A$. Therefore, $\alpha=t\circ i^A$ since $v$ is a monomorphism. Dually, we can prove Definition \ref{defn:pptc} (e).
 \end{proof}
 By Theorem \ref{thm:pretricat} and the above theorem, we have
\begin{corollary} Let $(\C, \mathcal{W}, \F)$ be a Hovey triple in the exact category $(\A, \E)$. Denote by $\X=\C\cap \mathcal{W}\cap \F$. Then $(\C\cap \F)/\X$ is a pre-triangulated category.
\end{corollary}

\subsection*{The homotopy theory of exact model categories}

Recall that in a model category $\A$, there are three classes of morphisms, called {\it cofibrations, fibrations} and {\it weak equivalences}. 
A morphism which is both a weak equivalence and a (co-)fibration is called a {\it trivial $($co-$)$fibration}. For details about model categories, we refer the reader to \cite{Quillen67} or \cite{Hovey99}. Given a model category $\A$, an object $A\in \A$ is called {\it cofibrant} if $0\to A$ is a cofibration, it is called {\it fibrant} if $A\to 0$ is a fibrantion, and it is called {\it trivial} if $0\to A$ is a weak equivalence. 

An {\it exact model category} is a {\it weakly idempotent complete} (i.e. every split monomorphism is an inflation) exact category which has a model structure such that the cofibrations are inflations with cofibrant cokernels and the fibrations are the deflations with fibrant kernels \cite[Definition 2.1]{Hovey02}. If $(\C,\mathcal{W},\F)$ is a Hovey triple in a weakly weakly idempotent complete exact category $(\A,\E)$, then by \cite[Theorem 2.2]{Hovey02} or \cite[Theorem 3.3]{Gillespie11}, there is an exact model structure $\M$ such that $\C$ is the class of cofibrants, $\F$ is the class of fibrants and $\mathcal{W}$ is the class of trivial objects. In this case, a weak equivalence is a morphism $f$ which factors as $f=p\circ i$ where $i$ is an inflation with a cokernel in $\C\cap \mathcal{W}$ and $p$ is a deflation with a kernel in $\mathcal{W}\cap \F$.

By \cite[Theorem VII.4.2]{Beligiannis/Reiten07}, if $\M$ is an exact model structure induced by a Hovey triple $(\C,\mathcal{W},\F)$ in an exact category $\A$, then the homotopy category $\Ho(\M)$ of $\M$ is equivalent to the subfactor category $(\C \cap \F)/\X$, where $\X=\C\cap \mathcal{W}\cap \F$ and $\Ho(\M)$ is the localization of the category $\A$ with respect to the weak equivalences. In this case, by \cite[Proposition I.5.1]{Quillen67} and \cite[Proposition 4.4]{Gillespie11}, the class of weak equivalences coincides with the class $\mathcal{S}$ defined as in (\ref{we}). By \cite[Proposition I.1.3.5]{Quillen67} (see also \cite[Subsection 6.3]{Hovey99}), the homotopy category $\Ho(\M)$ has a pre-triangulated structure induced by the model structure $\M$. By Theorem \ref{thm:homoHoveytriple}, the subfactor category $ (\C\cap \F)/\X$ also has a pre-triangulated structure induced by the pre-partial triangulated structure $(0,0, \rmL(\F), \R(\C), \X)$. The following result shows that the equivalence between $\Ho(\M)$ and $ (\C\cap \F)/\X$ preserves pre-triangulated structures.

\begin{theorem} \label{thm:model vs triple} Let $\M$ be the exact model structure induced by the Hovey triple $(\C,\mathcal{W},\F)$ in a weakly idempotent complete exact category $(\A,\E)$. Denote by $\X=\C\cap\mathcal{W}\cap \F$, then the homotopy category $\Ho(\M)$ is equivalent to the subfactor category $(\C\cap \F)/\X$ as pre-triangulated categories.
\end{theorem}
\begin{proof} \  Denote by $\gamma_\M\colon \A\to \Ho(\M)$ the localization of $\A$ with respect to the weak equivalences. Note that for each object $A$ in $\A$, the special $\C$-precover $r_A\colon Q(A)\to A$ in the proof of Lemma \ref{lem:exacthomotopytriple} is the cofibrant replacement in the model category $\A$, and the special $\F$-preenvelope $j^A\colon A\to R(A)$ is the fibrant replacement. Therefore, by Quillen's homotopy category theorem \cite[Theorem I.1.1]{Quillen67}, $\Hom_{\Ho(\M)}(A,B)=\Hom_{(\C\cap \F)/\X}(R(Q(A)),R(Q(B)))$, and the homotopy category of cofibrant objects $\Ho(\C)$ and the homotopy category of fibrant objects $\Ho(\F)$ are equivalent to $\Ho(\M)$ as full subcategories. Moreover, any morphism from $A$ to $B$ in $\Ho(\M)$ is represented as $\gamma_\M(r_B)\circ \gamma_\M(j^{Q(B)})^{-1}\circ \ul{f} \circ \gamma_\M(j^Q(A))\circ \gamma_\M(r_A)^{-1}$ for some morphism $f\colon R(Q(A))\to R(Q(B))$ in $\C\cap \F$.

We recall the construction of the right triangulated structure of $\Ho(\M)$ in \cite[Chapter I, Sections 2-3]{Quillen67}. For every cofibrant object $A\in \C$, recall that we have fixed a right $\C$-triangle $A\xrto{i^A}X^A\xrto{p^A} U^A\to 0$, then $(1_A,1_A)$ can be factored as $ A\oplus A\xrto{\left(\begin{smallmatrix}
1_A & 1_A \\
i^A& 0
\end{smallmatrix}\right)} A\oplus X^A\xrto{(1_A,0)} A .$ Since $\left(\begin{smallmatrix}
1_A & 1_A \\
i^A& 0
\end{smallmatrix}\right)$ is a cofibration and $(1_A,0)$ is a trivial fibration, we know that $A\oplus X^A$ is a cylinder object of $A$. Moreover, by (\ref{kappaf}), for each morphism $f\colon A\to B$ in $\C$, we have the following commutative diagram
\[\xy\xymatrixcolsep{2pc}\xymatrix@C28pt@R26pt{ A\oplus A\ar[r]^-{\left(\begin{smallmatrix}
1_A&1_A\\
i^A&0
\end{smallmatrix}\right)}\ar[d]_{\left(\begin{smallmatrix}
f & 0\\
0 & f
\end{smallmatrix}\right)} & A\oplus X^A \ar[r]^-{(p^A,0)}\ar[d]^{\left(\begin{smallmatrix}
f & 0\\
0 & \sigma^f
\end{smallmatrix}\right)} & U^A \ar[d]^{\kappa^f}\\
 B\oplus B\ar[r]^-{\left(\begin{smallmatrix}
1_B&1_B\\
i^B&0
\end{smallmatrix}\right)}& B\oplus X^B\ar[r]^-{(p^B,0)}& B\oplus B
}\endxy\]
Thus, by Quillen's construction, the restriction of the suspension functor $\Sigma^\M|_\C$ on  $\Ho(\C)$ is defined by sending an object $A$ to $\Coker\left(\begin{smallmatrix}
1_{A} & 1_{A}\\
i^{A} & 0
\end{smallmatrix}\right)=\Coker i^{A}=U^{A}$ and a morphism $\gamma_\M(f)\colon A\to B$ to $\gamma(\kappa^{f})$. So $\Sigma^\M|_{\C}=\Sigma^\X$ on $\C$. Let $f\colon A\to B$ be a cofibration in $\C$, i.e. $f$ is an inflation in $\C$ with a cokernel $g\colon B\to C$ in $\C$. Then the standard right triangle in $\Ho(\M)$ is defined to be the sequence $A\xrto{\gamma_\M(f)}B\xrto{\gamma_{\M}(g)}C\xrto{\gamma_\M(\xi(f,g))}\Sigma^\M(A)$, where $\xi(f,g)$ is defined as in the diagram (\ref{xif}) since the group coaction of $\Sigma^\M(A)=U^A$ on $C$ is given by the morphism $\xi(f,g)$ in this case (see \cite[Section I.3]{Quillen67} for details).

Now assume that $A\xrto{\ul{f}} B\xrto{\ul{g}} R(C)\xrto{R(\ul{h})}G^\X(A)$ is a standard right triangle in the subfactor category $(\C\cap \F)/\X$, i.e. it is induced by the right triangle $A\xrto{\ul{f}}B\xrto{\ul{g}}C\xrto{\ul{h}}\Sigma^\X(A)$ in $\C/\X$. Without loss of generality, we may assume that $A\xrto{\ul{f}}B\xrto{\ul{g}}C\xrto{\ul{h}}\Sigma^\X(A)$ is induced by the diagram (\ref{xif}), thus $A\xrto{\ul{f}} B\xrto{\ul{g}} R(C)\xrto{R(\ul{h})}G^\X(A)$ is a distinguished right triangle in $\Ho(\M)$ since it is isomorphic to the standard right triangle $A\xrto{\gamma_\M(f)}B\xrto{\gamma_{\M}(g)}C\xrto{\gamma_\M(h)}\Sigma^\M(A)$ in $\Ho(\M)$. Therefore, the embedding $(\C\cap \F)/\X\hookrightarrow \Ho(\M)$ preserves the right triangulated structures (we remind the reader that for any object $A$ in $(\C\cap \F)/\X$, $G^\X(A)=R(\Sigma^\X(A))$ which is isomorphic to $\Sigma^\M(A)=\Sigma^\X(A)$ via $\gamma_\M(j^{\Sigma^\X(A)})$ in $\Ho(\M)$).

Dually, we can prove that the embedding $(\C\cap \F)/\X\hookrightarrow \Ho(\M)$ preserves the left triangulated structures.
 \end{proof}

\end{document}